\documentclass[article]{amsart}
\usepackage{cases}
\usepackage{amsmath, amsfonts, pifont, amssymb, xcolor, mathrsfs}
\usepackage{verbatim}
\usepackage[bookmarks=true]{hyperref}
\usepackage[margin=1 in]{geometry}
\usepackage[numbers,sort]{natbib}
\usepackage{mathtools}
\mathtoolsset{showonlyrefs}
\newtheorem{theorem}{Theorem}[section]
\newtheorem{lemma}[theorem]{Lemma}
\newtheorem{proposition}[theorem]{Proposition}
\newtheorem{corollary}[theorem]{Corollary}

\newtheorem{conj}[theorem]{Conjecture}
\newcommand{\R}{\mathbb{R}}
\newcommand{\T}{\mathbb{T}}

\newcommand{\bS}{\mathbb{S}}
\newcommand{\bR}{\mathbb{R}}
\newcommand{\bZ}{\mathbb{Z}}

\newcommand{\bT}{\mathbb{T}}

\newcommand{\beq}{\begin{equation}}
\newcommand{\eeq}{\end{equation}}
\newcommand{\beqq}{\begin{equation*}}
\newcommand{\eeqq}{\end{equation*}}

\newcommand{\lea}{\lesssim}

\newcommand{\gea}{\gtrsim}

\newcommand{\w}{\widetilde}

\theoremstyle{definition}
\newtheorem{definition}[theorem]{Definition}

\theoremstyle{remark}
\newtheorem{remark}{Remark}[section]

\numberwithin{equation}{section}

\def \l {\left}
\def \r {\right}

\def \lea {\lesssim}
\def \ca { \mathbf{1}}

\numberwithin{equation}{section}

\allowdisplaybreaks[4]


\begin{document}

\address{Yangkendi Deng
\newline \indent Department of Mathematics and Statistics, Beijing Institute of Technology.
\newline \indent Key Laboratory of Algebraic Lie Theory and Analysis of Ministry of Education.
\newline \indent  Beijing, China. \vspace{-0.2cm}}
\email{dengyangkendi@bit.edu.cn\vspace{-0.1cm}}

\address{Yunfeng Zhang
\newline \indent Department of Mathematical Sciences, University of Cincinnati, 2815 Commons Way
\newline \indent Cincinnati, OH, U.S.\vspace{-0.2cm}}
\email{zhang8y7@ucmail.uc.edu\vspace{-0.1cm}}

\address{Zehua Zhao
\newline \indent Department of Mathematics and Statistics, Beijing Institute of Technology.
\newline \indent Key Laboratory of Algebraic Lie Theory and Analysis of Ministry of Education.
\newline \indent  Beijing, China. \vspace{-0.2cm}}
\email{zzh@bit.edu.cn}

\title[Bilinear eigenfunction estimate, anisotropic Strichartz estimate, and energy-critical NLS]{
Sharp bilinear eigenfunction estimate, anisotropic Strichartz estimate, and energy-critical NLS
}

\author{Yangkendi Deng, Yunfeng Zhang and Zehua Zhao}


\begin{abstract}

We establish sharp bilinear eigenfunction estimates for the Laplace–Beltrami operator on the standard three-sphere $\mathbb{S}^3$, eliminating the logarithmic loss that has persisted in the literature since the pioneering work of Burq, Gérard, and Tzvetkov over twenty years ago. This completes the theory of multilinear eigenfunction estimates on the standard spheres. Our approach relies on viewing $\mathbb{S}^3$ as the compact Lie group $\mathrm{SU}(2)$ and exploiting its representation theory. Motivated by applications to the energy-critical nonlinear Schr\"odinger equation (NLS) on $\mathbb{R} \times \mathbb{S}^3$, we also prove a refined anisotropic Strichartz estimate on the cylindrical space $\mathbb{R}_{x_1} \times \mathbb{T}_{x_2}$ of $L^\infty_{x_2}L^4_{t,x_1}$-type, adapted to certain spectrally localized functions. The argument relies on multiple sharp measure estimates and a robust kernel decomposition method. Combining these two key ingredients, we derive a refined bilinear Strichartz estimate on $\mathbb{R} \times \mathbb{S}^3$, which in turn yields small-data global well-posedness for the above mentioned NLS in the energy space.

\end{abstract}

\maketitle

\setcounter{tocdepth}{1}
\tableofcontents

\parindent = 10pt     
\parskip = 8pt

\section{Introduction and main results}

Let $\R$ denote the real line, and let $\bS^3$ denote the standard three-sphere. We study the initial value problem for the cubic nonlinear Schr\"odinger equation (NLS) on the product manifold $\mathbb{R} \times \mathbb{S}^3$, 
\begin{equation}\label{eq: maineq}
\begin{cases}
iu_t + \Delta u = \pm|u|^2 u, \\
u(0,x,y) = u_0(x,y),
\end{cases}
\end{equation}
where $u(t,x,y)$ is a complex-valued function on the spacetime $\R_t\times\R_x\times\bS^3_y$. 
For strong solutions $u$ of \eqref{eq: maineq}, we have energy conservation,
\begin{align}\label{eq: energy}
    E(u(t))=\frac12\int_{\R\times\bS^3} |\nabla u(t,x,y)|^2 \ {\rm d}x\ {\rm d}y\pm \frac14\int_{\R\times\bS^3}|u(t,x,y)|^4 \ {\rm d}x\ {\rm d}y = E(u_0),
\end{align}
and mass conservation,
\begin{align}\label{eq: mass}
    M(u(t))=\frac12\int_{\R\times\bS^3}|u(t,x,y)|^2\ {\rm d}x\ {\rm d}y=M(u_0). 
\end{align}
The above model, as the cubic NLS on a four-dimensional manifold, is referred to energy-critical since the energy of the cubic NLS on $\R^4$
is invariant under its natural scaling symmetry.

The main goal of this paper is to establish small-data global well-posedness for \eqref{eq: maineq} in the critical space, namely the energy space $H^1(\R\times\bS^3)$. Previously established energy-critical models in four dimensions include $\R^4$, $\mathbb{H}^4$, $\T^4$, $\R^m\times\T^{4-m}$, and in three dimensions include $\R^3$, $\mathbb{H}^3$, $\T^3$, $\R^m\times\T^{3-m}$, $\bS^3$, $\T\times\bS^2$; for references, see Tables \ref{tab:4D} and \ref{tab:3D}.\footnote{For the energy-critical NLS on higher-dimensional Euclidean spaces and tori, we refer to \cite{KK24,Kwa24,KV10,Vis07}. For mass-critical NLS, we refer to \cite{dodson2012global,dodson2015global,dodson2016global1,dodson2016global2}.} 

In particular, the breakthrough result of Herr, Tataru, and Tzvetkov on $\T^3$ \cite{HTT11} established the first instance of energy-critical well-posedness on a compact manifold, which also paved the way for the later study on other product manifolds such as $\R^m\times\T^{4-m}$.

There are key differences between the analysis of NLS on flat spaces such as Euclidean spaces and tori, and on positively curved compact manifolds such as spheres. Weaker dispersion for the Schr\"odinger equation, combined with the absence of a Fourier transform, greatly hinders the analysis on these latter manifolds. A notable example is the four-sphere $\bS^4$, which remains out of reach due to the failure of the $L^4$-Strichartz estimate as shown by Burq, G\'erard, and Tzvetkov \cite{BGT04}. 
In comparison, on $\R^m\times\T^{4-m}$, $L^p$-Strichartz estimates are available for $p<4$, which lay the foundation for the well-posedness theory. 
On the hybrid model $\mathbb{R} \times \mathbb{S}^3$, which couples Euclidean and spherical components, the $L^4$-Strichartz estimate is available, but no $L^p$-Strichartz estimate for $p<4$ is presently known, rendering the well-posedness theory delicate. 

The absence of a Fourier transform on a general compact manifold is first remedied by the spectral theory of the Laplace--Beltrami operator. For a waveguide manifold such as $\R\times\bS^3$, it is also clear that eigenfunctions of the Laplace--Beltrami operator on the compact factor, in our case $\mathbb{S}^3$, play an essential role in the analysis of NLS, as those are static solutions to the linear Schr\"odinger equation. 
Sogge established foundational $L^p$ estimates of eigenfunctions on compact manifolds \cite{Sog88}, which are sharp on spheres. However, for nonlinear analysis, interactions among eigenfunctions are equally if not more important. Such interactions are quantified in the pioneering works \cite{BGT05,BGT052} of Burq, G\'erard, and Tzvetkov in terms of bilinear and multilinear estimates. They have been highly valuable in the well-posedness theory of NLS on compact manifolds, especially spheres or product manifolds that have spherical factors. For example, using sharp bilinear eigenfunction estimates on $\mathbb{S}^2$,  
the authors proved in \cite{BGT05} uniform local well-posedness of the cubic NLS in the Sobolev space $H^s(\mathbb{S}^2)$, for the range $s>\frac14$ that is sharp except the endpoint
. Similarly, on $\mathbb{S}^3$ and $\T\times\bS^2$, trilinear eigenfunction estimates on $\mathbb{S}^3$ and $\mathbb{S}^2$ play a central role in Burq--Gérard--Tzvetkov’s proof of well-posedness for the sub-quintic NLS \cite{BGT052} in the energy space, and in Herr's and Herr--Strunk's later refinement establishing well-posedness for the quintic NLS \cite{Her13,HS15}. 

For the cubic NLS on $\R\times \bS^3$, bilinear eigenfunction estimates on the factor $\bS^3$ are vital. 
Let $f,g$ be eigenfunctions of the Laplace--Beltrami operator on $\bS^3$, with eigenvalues $-m(m+2)$, $-n(n+2)$ respectively, $m,n\in\mathbb{Z}_{\geq 0}$. Assume $m\geq n$. 
Then it was proved in \cite{BGT052} that 
$$\|fg\|_{L^2(\bS^3)}\leq C(n+1)^{\frac12}\log^{\frac12}(n+2) \|f\|_{L^2(\bS^3)}\|g\|_{L^2(\bS^3)},$$
where $C$ is a positive universal constant. 
A significant limitation of the above estimate lies in the logarithmic factor, which is not expected to be sharp. This becomes a more significant issue for the important question of critical well-posedness for \eqref{eq: maineq} in the energy space, for which sharp ``scale-invariant'' bilinear eigenfunction estimates would be needed. However, since it was first introduced, the above bilinear estimate has not been refined in the literature. The delicacy of this estimate stems from its $L^4$ nature, which corresponds to the critical breakpoint in Sogge’s $L^p$ eigenfunction bounds on $\bS^3$. 
In fact, among all spheres, the three-sphere is the only case for which a sharp multilinear eigenfunction estimate has been absent.

As a key contribution of this paper, we fill this gap. In Theorem \ref{mainthm: 1}, we eliminate the log factor and prove the \textit{sharp} scale-invariant bilinear eigenfunction estimate on $\bS^3$, which also implies all the sharp multilinear eigenfunction estimates. In contrast to the microlocal analytic methods in \cite{BGT052}, our approach is distinctly more algebraic and analytically transparent. It is based on viewing the standard three-sphere as the compact Lie group $\mathrm{SU}(2)$ and exploiting the associated representation theory. As is well understood, products of eigenfunctions can be expressed as linear combinations of matrix entries of tensor products of irreducible representations. To estimate these quantities, we work within the framework of the Clebsch–Gordan coefficients, which dictates the decomposition of tensor products into irreducible representations. Exploiting structural properties of these coefficients, we are able to reduce the bilinear eigenfunction estimates to elementary bounds, illustrating how representation-theoretic tools can substantially clarify analytic questions in partial differential equations. Moreover, this approach extends naturally to spheres of arbitrary dimension and yields sharp bilinear eigenfunction estimates in a unified algebraic framework; see Remark~\ref{generalization}.

In order to treat the critical well-posedness of \eqref{eq: maineq} on $\R\times\bS^3$, however, obtaining the sharp bilinear eigenfunction estimate on $\bS^3$ is only a necessary step in our approach. To carry out the strategy essentially devised by Herr \cite{Her13} and later by Herr and Strunk \cite{HS15}, we also need a refined anisotropic Strichartz estimate on $\R_{x_1}\times\T_{x_2}$ of $L^\infty_{x_2}L^p_{t,x_1}$-type, tailored to certain spectrally localized data arising from an almost orthogonality argument. The appearance of the $\T$ component is due to the fact that shifting the spectrum of the Laplace–Beltrami operator on $\bS^3$ by $-1$ yields the spectrum of $\T$, up to the removal of the zero mode.
Similar to the discussion by Herr, Tataru, and Tzvetkov in \cite{HTT14}, there are two possible approaches to this problem. One is to get an $L^\infty_{x_2}L^p_{t,x_1}$-type Strichartz estimate on $\R_{x_1}\times\T_{x_2}$ for some $p<4$, which currently seems out of reach. The other is to refine reasoning at the $L^4$ level, as was done in \cite{HTT14}. We also take the second approach here, and succeed in establishing the estimate stated in Theorem~\ref{mainthm: 2} through a careful combination of sharp counting and measure estimates. A key technical component is a novel kernel decomposition method, which facilitates the application of distinct arithmetic--geometric mean inequalities to the quadrilinear form---an essential step in assembling the various precise measure estimates. Notably, these techniques are sufficiently robust to yield \textit{sharp} $L^4$-Strichartz estimates for the \textit{hyperbolic Schr\"odinger equation} on $\mathbb{R}\times\mathbb{T}$; see Remark~\ref{rmk 5.1}. Combining Theorem \ref{mainthm: 2} with Theorem \ref{mainthm: 1}, we prove a refined bilinear Strichartz estimate on $\R\times \bS^3$ as recorded in Theorem \ref{mainthm: 3}. As a consequence, we obtain the small-data global well-posedness for the Cauchy problem \eqref{eq: maineq}, in Theorem \ref{mainthm: 4}.

Our approach is inherently interdisciplinary, blending techniques from representation theory, Fourier analysis, number theory, and nonlinear PDEs. This fusion not only resolves the problem at hand, but also illustrates the profound influence of algebraic and geometric structures on dispersive dynamics. Let us compare the cubic NLS on $\R\times\bS^3$ with some other energy-critical models that have not been treated in the literature. The analysis of the quintic NLS on $\R\times\bS^2$ is in fact considerably simpler, since one can rely on the trilinear eigenfunction bound on $\bS^2$ together with an $L^\infty_{x_2}L^p_{t,x_1}$-type Strichartz estimate on $\R_{x_1}\times\T_{x_2}$ valid for $p<6$ (in particular, for $p=4$); see Remark \ref{rem: RmSn}. 
For the cubic NLS on $\T \times \bS^3$, the question remains unresolved. In light of the sharp bilinear eigenfunction estimate on $\bS^3$, small-data global well-posedness would follow from a refined anisotropic Strichartz estimate on $\T\times\T$, 
analogous to Theorem \ref{mainthm: 2}. We leave this problem for future work. 
Finally, the cubic NLS on $\mathbb{R}^{m}\times \mathbb{T}^{2-m}\times \mathbb{S}^2$, $m=0,1,2$, and  on $\bS^2\times\bS^2$, appear to be the most difficult cases. As there is no scale-invariant bilinear eigenfunction estimate on $\mathbb{S}^2$, a substantially different strategy would be required; see also Remark~\ref{rem: R2S2}.

\subsection{Statement of main results}
We now present precisely the main results of this paper. 

\begin{theorem}[Sharp bilinear eigenfunction estimate]\label{mainthm: 1}
For $m,n\in\mathbb{Z}_{\geq 0}$, let $f,g$ be eigenfunctions of the Laplace--Beltrami operator $\Delta_{\bS^3}$ on $\bS^3$ such that 
    $$\Delta_{\bS^3} f=-m(m+2)f,\ \ \ \Delta_{\bS^3} g=-n(n+2)g. $$
    Assume that $m\geq n$.
    Then 
    $$\|fg\|_{L^2(\bS^3)}\leq C  (n+1)^{\frac12}\|f\|_{L^2(\bS^3)}\|g\|_{L^2(\bS^3)}.$$    
\end{theorem}

\begin{remark}
This sharp bilinear eigenfunction estimate immediately yields the corresponding sharp trilinear and general multilinear eigenfunction estimates, improving upon (1.7) and (1.8) of \cite{BGT052}; see Corollary \ref{cor: Trilinear}. These multilinear estimates refine the corresponding linear eigenfunction bounds originally established by Sogge \cite{Sog88}.
\end{remark}

\begin{remark}
Although $\bS^3$ is expected to exhibit the largest possible growth of Laplace–Beltrami eigenfunctions among all three-dimensional compact manifolds, it remains open to establish the same bilinear eigenfunction estimates on general three-dimensional compact manifolds, in particular on general Zoll manifolds.
\end{remark}

Next we state our refined Strichartz estimate on $\R\times\T$.  
We will need a good Schwartz function to replace the characteristic function of the unit interval. Throughout this paper, let $\varphi(t)\in \mathcal{S}(\R)$ be such that: (1) $\widehat{\varphi}(\tau)\geq 0$ for all $\tau\in\R$; (2) the support of $\widehat{\varphi}$ lies in $[-1,1]$; (3)  $\varphi(t)\ge 0$ for $t\in \bR$, and $ \varphi(t)\ge 1$ for $t\in[0,1]$. The existence of the function $\varphi$ is straightforward; see Lemma 1.26 in \cite{Dem20}.

\begin{theorem}[Refined anisotropic Strichartz estimate]\label{mainthm: 2}
Let $1\le M \le N$, $\delta\in (0,\frac18)$. Let $a\in\bR^2$ with $|a|=1$, and $c\in \R$. Let $\xi^0\in \bR\times \bZ$. Define
$$\mathcal{R}=\{\xi=(\xi_1,\xi_2)\in \bR\times \bZ : |\xi-\xi^0|\le N, |a\cdot \xi-c|\le M  \}.$$
Assume that $\phi \in L^2(\bR\times \bT)$ and $\text{supp}(\widehat{\phi})\subset \mathcal{R}$.  Then the following holds:
\begin{align*}
   &  \l\|\varphi(t)\int_{\bR \times \bZ} e^{ix_1\cdot \xi_1-it|\xi|^2} \widehat{\phi}(\xi)\ {\rm d}\xi\r\|_{ L^4_{t,x_1}(\bR\times \bR)} \leq C \l(\frac{M}{N}\r)^{\delta} N^{\frac14}   \|\phi\|_{L^2(\bR\times \bT)},
\end{align*}
uniformly in $a\in\bR^2$ with $|a|=1$, $c\in \R$, $\xi^0\in \R\times\bZ$, and $1\le M\le N$.
\end{theorem}

\begin{remark}\label{rem: LinftyL4}
 In particular, by choosing $M=N$ and $c=0$, we obtain the $L^\infty_{x_2}L^4_{t,x_1}$-type Strichartz estimate on $\R_{x_1}\times \bT_{x_2}$:
\begin{align}\label{eq: res}
   \sup_{\xi^0\in \R\times\bZ} \l\|\varphi(t)\int_{ 
    \substack{\xi\in\bR \times \bZ \\
    |\xi-\xi^0|\leq N}
    } e^{ix_1\cdot \xi_1-it|\xi|^2} \widehat{\phi}(\xi)\ {\rm d}\xi\r\|_{L^4_{t,x_1}( \R\times \bR)} \leq C N^{\frac14} \|\phi\|_{L^2}.
\end{align}   

By Bernstein's inequality on $\bT$, the case $\xi^0=0$ of the above estimate is also a consequence of the  $L^4$-Strichartz estimate on $\R\times \bT$ established in \cite{TT01}.
\end{remark}

Based on Theorem \ref{mainthm: 1} and Theorem \ref{mainthm: 2}, we have the following refined bilinear Strichartz estimate on $\R\times \bS^3$, a crucial ingredient for the well-posedness theory of \eqref{eq: maineq} in the energy space. 

\begin{theorem}[Refined bilinear Strichartz estimate]\label{mainthm: 3}

    For $1\leq N_2\leq N_1$ and $0<\delta<\frac{1}{8}$, we have 
    $$\|e^{it\Delta}P_{N_1}f\cdot e^{it\Delta}P_{N_2}g\|_{L^2([0,1]\times\R\times\bS^3)}\leq C  N_2\left(\frac{N_2}{N_1}+\frac{1}{N_2}\right)^\delta\|f\|_{L^2(\R\times\bS^3)}\|g\|_{L^2(\R\times\bS^3)}.$$
\end{theorem}

\begin{remark}
    In particular, by choosing $N_1=N_2=N\geq 1$, we get the $L^4$-Strichartz estimate on $\R\times\bS^3$
    \begin{align}\label{eq: L4 Strichartz R times T}
        \|e^{it\Delta}P_{N}f\|_{L^4([0,1]\times\R\times\bS^3)}\leq C  N^{\frac12}\|f\|_{L^2(\R\times\bS^3)}.
    \end{align}
\end{remark}

\begin{remark}
The above refined bilinear Strichartz estimates also hold on $\R^m\times \T^{4-m}$, $m=0,1,2,3$, as established by Herr, Tataru, and Tzvetkov \cite{HTT14}, Ionescu and Pausader \cite{IP12CMP}, and Bourgain \cite{Bou13}. 
\end{remark}

Finally, we present our well-posedness result for \eqref{eq: maineq}. Let $B_\varepsilon(\phi):=\{u_0\in H^1(\R\times\bS^3): \|u_0-\phi\|_{H^1}<\varepsilon\}$. 
\begin{theorem}[Well-posedness]\label{mainthm: 4}
Let $s\geq 1$. For every $\phi\in H^1(\R\times\bS^3)$, there exists $\varepsilon>0$ and $T=T(\phi)>0$, such that for all initial data $u_0\in B_\varepsilon(\phi)$, the Cauchy problem \eqref{eq: maineq} has a unique solution
\[
u\in C([0,T);H^s(\R\times\bS^3))\cap X^s([0,T)). 
\]
This solution obeys conservation laws \eqref{eq: energy} and \eqref{eq: mass}, and the flow map 
$$B_\varepsilon(\phi)\cap H^s(\R\times\bS^3)\ni u_0\mapsto u\in C([0,T);H^s(\R\times\bS^3))\cap X^s([0,T))$$
is Lipschitz continuous. 
Moreover, there exists a constant $\eta_0>0$ such that if $\|u_0\|_{H^s(\R\times\bS^3)}\le\eta_0$ then the solution extends globally in time.
\end{theorem}

The function spaces $X^s([0,T))$ used to construct the solution in the above theorem, 
namely those in Definition \ref{def: Xs Ys}, are similar to the ones used in \cite{Her13} and \cite{HS15}, which are based on the dyadic Littlewood--Paley projections, and the $U^p,V^p$ spaces introduced in \cite{KT05} (see also \cite{MR1827277}).

\begin{remark}[On the large-data global well-posedness]
A natural and important question is whether the small-data global well-posedness result can be extended to the large-data setting for both the defocusing and focusing nonlinearities. 
As observed in previous large-data works on manifolds such as \cite{IP12duke,IP12CMP,PTW14}, at least for the defocusing case, the large-data global well-posedness and scattering results on Euclidean spaces \cite{Iteam08,RV07} can be treated as a black box, and it would then suffice to establish several key properties of the linear Schr\"odinger flow such as Lemmas 4.3 and 7.3 of \cite{IP12duke}. However, the current absence of any scale-invariant $L^p$-Strichartz estimates on $\mathbb{R}\times\mathbb{S}^3$ for $p<4$ creates substantial obstacles to proving those properties; see Section \ref{subsec: Strichartz on RS3} for further discussions.

\end{remark}
\begin{table}[htbp]
\centering
\caption{Global well-posedness for 4D energy-critical NLS models in the energy space}
\label{tab:4D}
\begin{tabular}{|c|c|c|c|c|}
\hline
 \textbf{Geometry} & \textbf{Small data} &  \textbf{Large data}\\

\hline
$\mathbb{R}^4$   & Cazenave--Weissler \cite{CW89} & Ryckman--Visan \cite{RV07},  
Dodson \cite{MR3940908} \\
\hline
$\mathbb{H}^4$   & Anker--Pierfelice \cite{AP09} &  {\it Open}\\
\hline
$\mathbb{T}^4$   & Herr--Tataru--Tzvetkov \cite{HTT14}, 
Bourgain \cite{Bou13} & Yue \cite{Yue21} \\
\hline
$\mathbb{R}\times \mathbb{T}^3$  & \multicolumn{2}{|c|}{Ionescu--Pausader \cite{IP12CMP} } 
  \\
\hline
$\mathbb{R}^2\times \mathbb{T}^2$ &  Herr--Tataru--Tzvetkov \cite{HTT14} & Zhao \cite{Zhao19}\\
\hline
$\mathbb{R}^3\times \mathbb{T}$  & Herr--Tataru--Tzvetkov \cite{HTT14} & Zhao \cite{Zhao21}\\
\hline
$\mathbb{R}\times \mathbb{S}^3$   & {\it Current paper} & {\it Open 
} \\
\hline
\end{tabular}
\end{table}

\begin{table}[htbp]
\centering
\caption{Global well-posedness for 3D energy-critical NLS models in the energy space}
\label{tab:3D}
\begin{tabular}{|c|c|c|c|c|}
\hline
 \textbf{Geometry} & \textbf{Small data} & \textbf{Large data} \\
\hline
$\mathbb{R}^3$   & Cazenave--Weissler \cite{CW89} & Colliander--Keel--Staffilani--Takaoka--Tao \cite{Iteam08}
\\
\hline
$\mathbb{H}^3$   & Anker--Pierfelice \cite{AP09} & Ionescu--Pausader--Staffilani \cite{IPS12} \\
\hline
$\mathbb{T}^3$   & Herr--Tataru--Tzvetkov \cite{HTT11} & Ionescu--Pausader \cite{IP12duke} \\
\hline
$\mathbb{R}\times \mathbb{T}^2$ &\multicolumn{2}{|c|}{Hani--Pausader \cite{HP14} }  \\
\hline
$\mathbb{R}^2\times \mathbb{T}$ & \multicolumn{2}{|c|}{ Zhao \cite{Zhao21} }   \\
\hline
$\mathbb{S}^3$  & Herr \cite{Her13} & Pausader--Tzvetkov--Wang \cite{PTW14}\\
\hline
$\mathbb{T}\times \mathbb{S}^2$ & Herr--Strunk \cite{HS15} & {\it Open} \\
\hline
\end{tabular}
\end{table}

\subsection{Organization of the paper}
In Section \ref{pre}, we review the Plancherel and Littlewood--Paley theory for $\R\times\bS^3$, 
and collect some basic estimates such as the Bernstein and Sobolev inequalities. In Section \ref{3}, we review the representation theory of $\mathrm{SU}(2)$ that is essential for our analysis, especially the framework of Clebsch--Gordan coefficients. In Section \ref{sec: bilinear}, we prove the sharp bilinear eigenfunction estimate on $\mathbb{S}^3$ (Theorem \ref{mainthm: 1}). 
In Section \ref{sec:  Strichartz}, we prove the refined anisotropic Strichartz estimate on $\mathbb{R} \times \mathbb{T}$ (Theorem \ref{mainthm: 2}).   
In Section \ref{bilinear and WP}, we combine the previous results to establish the refined bilinear Strichartz estimate on $\mathbb{R} \times \mathbb{S}^3$ (Theorem \ref{mainthm: 3}) and prove well-posedness for the energy-critical NLS (Theorem \ref{mainthm: 4}). Finally, in Section \ref{7}, we 
discuss related open problems such as the optimal $L^\infty_{x_2}L^p_{t,x_1}$-type Strichartz estimate on $\mathbb{R}_{x_1} \times \mathbb{T}_{x_2}$ and the Strichartz estimate on $\R\times\bS^3$.

\subsection{Notation}
 
We write $A \lesssim B$ if $A \leq C B$ for some absolute constant $C>0$. (Without ambiguity, we may use $C$ to denote a positive absolute constant, whose value may change from line to line.) We write $A \sim B$ if both $A \lesssim B$ and $B \lesssim A$ hold.

We write $A\ll B$ if there exists a sufficiently small constant $c > 0$ such that $A\leq cB$.
We use the usual $L^{p}$ spaces and Sobolev spaces $H^{s}$. For $1\leq p,q\leq\infty$, we use $L_x^pL_y^q$ to denote mixed-norm Lebesgue spaces such that 
$$\|f\|_{L^p_xL^q_y}:=\left(\int \left(\int |f(x,y)|^q\ {\rm d}y\right)^{\frac{p}{q}}\ {\rm d}x\right)^{\frac1p}.$$
Our notation for the Fourier transform on $\R$ is 
$$\widehat{f}(\xi)=\frac1{2\pi}\int_\R f(x) e^{-i\xi x }\ {\rm d}x, \ \xi\in\R.$$
Our notation for the Fourier transform on $\R_{x_1}\times \bT_{x_2}$ is 
$$\widehat{f}(\xi_1,\xi_2)=\frac1{4\pi^2}\int_0^{2\pi}\int_{\R}f(x_1,x_2) e^{-i(\xi_1 x_1+ \xi_2x_2)}\ {\rm d}x_1\ {\rm d}x_2, \ (\xi_1,\xi_2)\in\R\times\bZ.$$
\subsection*{Acknowledgments} 
We 
sincerely 
thank Professors Nicolas Burq, Sebastian Herr, and Nikolay Tzvetkov for their many helpful suggestions and insightful discussions. We warmly thank Professor Herbert Koch and Professor Daniel Tataru for their careful reading of the paper, which helped improve its clarity. 
Y. Deng was supported by China Postdoctoral Science Foundation (Grant No. 2025M774191) and the NSF grant of China (No. 12501117). 
Y. Zhang gratefully acknowledges the hospitality of both the Beijing Institute of Technology and the University of Cincinnati, where this work was conceived and carried out. 
Z. Zhao was supported by the National Key R\&D Program of China (2025YFA1018500), the NSF grant of China (No. 12426205, 12271032), Beijing Natural Science Foundation (No. 1262019) and the Beijing Institute of Technology Research Fund Program for Young Scholars.
\section{Preliminaries}\label{pre}
\subsection{Spectral theory, and Littlewood--Paley projectors}
Let $\Delta_\R$ and $\Delta_{\bS^3}$ denote the standard Laplace--Beltrami operators on $\R$ and $\bS^3$ respectively, and take $\Delta=\Delta_{\R}+\Delta_{\bS^3}$ as the Laplace--Beltrami operator on $\R\times\bS^3$.

The joint spectral decomposition of $\Delta_{\R}$ and $\Delta_{\bS^3}$  takes the following form. For $f\in L^2(\R\times\bS^3)$, 
\begin{align}\label{eq: spec decomp 1}
  f(x,y):=\int_{\R}\sum_{k=0}^\infty f_{\omega,k}(y) e^{i\omega x}\ {\rm d}\omega,\ x\in\R, \ y\in \bS^3,
\end{align}
where each $f_{\omega,k}$ is an eigenfunction of $\Delta_{\bS^3}$ such that 
$$\Delta_{\bS^3}f_{\omega,k}=-k(k+2)f_{\omega,k}.$$
Here ${\rm d}\omega$ denotes the standard Lebesgue measure on $\R$. 
We may also rewrite \eqref{eq: spec decomp 1} as
\begin{align}\label{eq: spec decomp 2}
  f(x,y):=\int_{\R\times\bZ_{\geq 0}} f_{\omega,k}(y) e^{i\omega x}\ {\rm d}\omega\ {\rm d}k,\ x\in\R, \ y\in \bS^3,
\end{align}
where ${\rm d}k$ denotes the counting measure on $\bZ$. Note that 
$$\Delta (f_{\omega,k}(y) e^{i\omega x})= [-\omega^2-(k+1)^2+1](f_{\omega,k}(y) e^{i\omega x}),$$
which, together with the above decomposition of $L^2(\R\times \bS^3)$, gives an explicit functional calculus for $\Delta$. We also mention that in the subsequent treatment of Strichartz estimates on $\R\times\bS^3$, we will shift the standard Laplace--Beltrami operator $\Delta$ to $\Delta-\mathrm{Id}$, which has the cleaner-looking spectrum $-\omega^2-(k+1)^2$, $(\omega,k)\in\R\times\bZ_{\geq 0}$. In light of this and for convenience, we fix the following terminology. 
\begin{definition}
    Given the above spectral decomposition \eqref{eq: spec decomp 1} or \eqref{eq: spec decomp 2} of $f\in L^2(\R\times \bS^3)$, we name
    $$(\xi_1,\xi_2):=(\omega,k+1)\in \R\times\bZ_{\geq 1}$$ as the spectral parameters. 
    For any bounded subset $A$ of $\R\times\bZ$, we say $f$ is spectrally supported in $A$ if $f_{\omega,k}=0$ for all $(\omega,k+1)\notin A$. We also define the spectral projector  
    $$P_{A}f(x,y):=\int_{\substack{
(\omega,k)\in \R\times\bZ_{\geq 0}\\
 (\omega,k+1)\in A  }}f_{\omega,k}(y) e^{i\omega x}\ {\rm d}\omega\ {\rm d}k. $$
\end{definition}

The Plancherel identity is 
$$\|f\|_{L^2(\R\times\bS^3)}^2=2\pi\int_{\R}\sum_{k=0}^\infty \|f_{\omega,k}\|_{L^2(\bS^3)}^2\ {\rm d}\omega.$$
We define the Sobolev norm 
$$\|f\|_{H^s(\R\times\bS^3)}^2:=\|(1-\Delta)^{\frac s2}f\|^2_{L^2(\R\times\bS^3)}
=\int_{\R}\sum_{k=0}^\infty(1+k(k+2)+\omega^2)^{\frac s2} \|f_{\omega,k}\|_{L^2(\bS^3)}^2.$$

Next, we define the standard Littlewood--Paley projectors associated with $\Delta$. Let us fix a nonnegative bump function $\beta\in C^\infty_0((\frac12,2))$ such that 
$$\sum_{m=-\infty}^\infty \beta(2^{-m}s)=1, \ s>0.$$
Then we set 
$\beta_0(s)=1-\sum_{m=1}^\infty \beta(2^{-m}s)\in C_0^\infty(\R_{>0})$ and $\beta_m(s)=\beta(2^{-m}s)$ for $m\geq1$. For $N=2^m$ with $m\geq 0$, 
define
$$P_{N}f:=\beta_m(\sqrt{-\Delta})f=\int_{\R}\sum_{k=0}^\infty\beta_m(\sqrt{k(k+2)+\omega^2}) f_{\omega,k}(y) e^{i\omega x}\ {\rm d}\omega,$$
and  
$$P_{\leq N}f:=\sum_{n=0}^m P_{2^n}f.$$

We end this subsection with the following important lemma on the spectral support of a product of two functions. 
\begin{lemma}\label{lem: fg spec supp}
Let $A$ be a bounded subset of $\R\times\bZ$. 
Let $N_2=2^m\geq 1$. Let $f,g\in L^2(\R\times \bS^3)$. 
Then $P_Af\cdot P_{N_2}g$ is spectrally supported in $A+[-2N_2,2N_2]^2$.     
\end{lemma}
\begin{proof}
    Write 
    $$P_Af=\int_{\substack{(\omega_1,k_1)\in \R\times\bZ_{\geq 0}\\(\omega_1,k_1+1)\in A}}f_{\omega_1,k_1}(y)e^{i\omega_1 x}\ {\rm d}\omega_1\ {\rm d}k_1,$$
    $$P_{N_2}g=\int_{(\omega_2,k_2)\in \R\times\bZ_{\geq 0}}\beta_m \l(\sqrt{(k_2+1)^2+\omega_2^2-1}\r)g_{\omega_2,k_2}(y)e^{i\omega_2 x}\ {\rm d}\omega_2\ {\rm d}k_2.$$
    Then 
    \begin{align*}
&P_Af\cdot P_{N_2}g\\
=& \int\!\!\!\int_{\substack{(\omega_i,k_i)\in \R\times\bZ_{\geq 0}, i=1,2\\(\omega_1,k_1+1)\in A}} 
\beta_m \l(\sqrt{(k_2+1)^2+\omega_2^2-1}\r)f_{\omega_1,k_1}(y)g_{\omega_2,k_2}(y)e^{i(\omega_1+\omega_2) x} \ {\rm d}\omega_1\ {\rm d}k_1\ {\rm d}\omega_2\ {\rm d}k_2.
    \end{align*}
    By Lemma \ref{lem: fg G supp} below, we may write 
    $$f_{\omega_1,k_1}(y)g_{\omega_2,k_2}(y)=\sum_k h_{\omega_1,k_1,\omega_2,k_2;k}(y),$$
    where $h_{\omega_1,k_1,\omega_2,k_2;k}$ is an eigenfunction of $\Delta_{\bS^3}$ with eigenvalue $-k(k+2)$, and $k$ ranges over $|k_1-k_2|,|k_1-k_2|+2,\ldots,k_1+k_2$. The above two identities imply that $P_Af\cdot P_{N_2}g$ is spectrally supported in the region of $(\omega,k+1)$ defined by 
    $$
    \left\{
    \begin{array}{l}
       \omega  =\omega_1+\omega_2,  \\
        (\omega_1,k_1+1)\in A, \\
        (k_2+1)^2+\omega_2^2-1\leq (2N_2)^2, \\
        k\in\{|k_1-k_2|,|k_1-k_2|+2,\ldots,k_1+k_2\}.
    \end{array}
    \right.
    $$
From the above conditions, it follows that $(\omega,k+1)\in A+[-2N_2,2N_2]^2$, which completes the proof.     
\end{proof}

\subsection{The Bernstein and Sobolev inequalities}
We briefly review the standard Bernstein and Sobolev inequalities on $\R\times\bS^3$ that are needed later. 
\begin{lemma}\label{lem: Bernstein}
For $1\leq q\leq p\leq\infty$, we have 
    $$\|P_{\leq N}f\|_{L^p(\R\times \bS^3)}\lesssim N^{4(\frac1q-\frac1p)}\|f\|_{L^q(\R\times \bS^3)}.$$
\end{lemma}
\begin{proof}
We observe that the individual spectra of $\Delta_\R$ and $\Delta_{\bS^3}$ in $P_{\leq N}f$ are both bounded by $\lesssim N$. Then we may apply the individual Bernstein's inequalities on $\R$ and $\bS^3$ (for the latter we refer to Corollary 2.2 of \cite{BGT04} which works on any compact manifold), and the Minkowski's inequality, to obtain  
    \begin{align*}
        \|P_{\leq N}f\|_{L^p(\R\times \bS^3)}
        &\lesssim N^{3(\frac1q-\frac1p)} \|P_{\leq N}f\|_{L^p(\R, L^q( \bS^3))}\\
        &\lesssim  N^{3(\frac1q-\frac1p)} \|P_{\leq N}f\|_{ L^q( \bS^3,L^p(\R))}\\
        &\lesssim N^{4(\frac1q-\frac1p)} \|f\|_{L^q(\R\times \bS^3)}. 
    \end{align*}
\end{proof}

\begin{lemma}\label{lem: Sob}
We have the embedding 
$H^1(\R\times\bS^3)\hookrightarrow L^4(\R\times \bS^3)$. 
\end{lemma}
\begin{proof}
This follows from the standard partition-of-unity argument that gives the same Sobolev estimates on compact manifolds as on Euclidean spaces. Namely, we cover $\bS^3$ by finitely many Euclidean patches $U_i\cong \R^3$, associated to which are a partition of unity $\sum_i\rho_i=1$. 
For $f\in L^4(\R\times \bS^3)$, we may estimate 
$$\|f\|_{L^4(\R\times\bS^3)}\leq\sum_i\|\rho_i f\|_{L^4(\R\times U_i)}.$$
As $\R\times U_i\cong \R^4$, we may apply standard Sobolev estimates on $\R^4$ to obtain  $\|\rho_i f\|_{L^4(\R\times U_i)}\lesssim \|\rho_if\|_{H^1(\R\times U_i)}\lesssim \|f\|_{H^1(\R\times\bS^3)}$, which finishes the proof. 
\end{proof}

\section{Analysis on the group $\mathrm{SU}(2)$}\label{3}
In this section, we collect useful information on analysis of the group $\mathrm{SU}(2)$. We follow introductory textbooks on Lie groups such as \cite{Far08} and \cite{Hal15}. Let $G$ denote the compact Lie group 
$$\text{SU}(2):=\left\{
\begin{pmatrix}
    a & b\\
    -\overline{b} & \overline{a}
\end{pmatrix}: a,b\in\mathbb{C}, |a|^2+|b|^2=1
\right\}.$$
The diffeomorphism between $\mathrm{SU}(2)$ and the standard three-sphere $\mathbb{S}^3$ is immediate. Moreover, the standard probability measure on $\mathbb{S}^3$ coincides with the normalized Haar measure $\mu$ on $\mathrm{SU}(2)$, using which one defines the Lebesgue spaces such as $L^2(G)$. 

\subsection{Irreducible representations and their tensor products}
The equivalence classes of irreducible representations of $\mathrm{SU}(2)$ are in one-to-one correspondence with the set of nonnegative integers. For each nonnegative integer $m$, let $\pi_m$ be the corresponding irreducible representation of $\mathrm{SU}(2)$ acting on the vector space $V_m$. We have  
$$\dim(V_m)=m+1.$$ 
Let $\langle\ ,\ \rangle$ denote an inner product (unique up to scalars) on $ V_m$ that is $\pi_m$-invariant, i.e.,  
$$\langle \pi_m(\mathrm{g}) p,\pi_m(\mathrm{g})q\rangle=\langle p,q\rangle,\ \mathrm{g}\in G,\ p,q\in  V_m.$$

Consider tensor products of representations. For $m,n\in\mathbb{Z}_{\geq 0}$, the tensor product $\pi_m\otimes \pi_{n}$ of the representations $\pi_m$ and $\pi_{n}$ is defined using 
$$(\pi_m\otimes \pi_{n})(\mathrm{g})(v_m\otimes v_{n})=(\pi_m(\mathrm{g})v_m)\otimes (\pi_n(\mathrm{g})v_{n}),\ \mathrm{g}\in G, \ v_m\in  V_m, \ v_n\in V_n.$$
The following fundamental theorem describes decomposition of the tensor product $\pi_m \otimes \pi_n$ into irreducible representations.

\begin{theorem}[Clebsch--Gordan decomposition] \label{thm: CG decomp}
    For each $m,n\in\mathbb{Z}_{\geq 0}$ with $m\geq n$, there exists a unitary isomorphism of $\mathrm{SU}(2)$-representations
    \begin{align*}
    \pi_m\otimes\pi_{n}\cong \bigoplus_{k\in\{m-n,m-n+2,\ldots,m+n\}}\pi_k.
\end{align*}
\end{theorem}
A proof can be found in Appendix C of \cite{Hal15}. Now we reformulate this theorem in terms of bases as follows. 

\begin{theorem}[Clebsch--Gordan coefficients]\label{thm:CG}
Let $m,n\in\mathbb{Z}_{\geq 0}$ with $m\geq n$. Pick any orthonormal basis 
$$\{v_{m,\alpha}\in  V_m:\alpha=-m,-m+2,\ldots,m\}$$ 
of $ V_m$, and 
any orthonormal basis 
$$\{v_{n,\beta}\in  V_n:\beta=-n,-n+2,\ldots,n\}$$  
of $ V_n$. Then there exists an orthonormal basis 
$$\{u_{k,\gamma}\in  V_m\otimes V_n: k=m+n,m+n-2,\ldots, m-n; \gamma=-k,-k+2,\ldots,k\}$$
of $V_m\otimes V_n$, 
such that the following hold. 

\noindent (1) For each $k=m-n,m-n+2,\ldots,m+n$, the restriction of the tensor product representation to the linear span of $\{u_{k,\gamma}\in  V_m\otimes V_n:\gamma=-k,-k+2,\ldots,k\}$ is isomorphic to the irreducible representation $\pi_k$. 

\noindent (2) 
Define the Clebsch--Gordan coefficients $C_{m,\alpha;n,\beta}^{k,\gamma}$ by\footnote{In our convention, the weight parameters of the Clebsch–Gordan coefficients are twice those used in \cite{Vil68,VMK88}. For additional properties of Clebsch–Gordan coefficients, see Chapter III of \cite{Vil68} and Chapter 8 of \cite{VMK88}.
} 
$$u_{k,\gamma}=\sum_{\alpha,\beta}C_{m,\alpha;n,\beta}^{k,\gamma}v_{m,\alpha}\otimes v_{n,\beta}.$$
Then they satisfy the orthogonality relations 
    \begin{equation}\label{eq: CG Ort 1}
\sum_{k,\gamma}C^{k,\gamma}_{m,\alpha;n,\beta}\overline{C^{k,\gamma}_{m,\alpha';n,\beta'}}=\delta_{\alpha,\alpha'}\delta_{\beta,\beta'},
    \end{equation}
    \begin{equation}\label{eq: CG Ort 2}
\sum_{\alpha,\beta}C^{k,\gamma}_{m,\alpha;n,\beta}\overline{C^{k',\gamma'}_{m,\alpha;n,\beta}}=\delta_{k,k'}\delta_{\gamma,\gamma'}.
    \end{equation}
\end{theorem}
\begin{proof}
In light of Theorem \ref{thm: CG decomp}, it suffices to explain part (2). This follows from the fact that $\{v_{m,\alpha}\otimes v_{n,\beta}: \alpha=-m,-m+2,\ldots,m;\beta=-n,-n+2,\ldots,n\}$ is an orthonormal basis of $V_m\otimes V_n$, and so the transition matrix between this basis and the other basis $\{u_{k,\gamma}\}_{k,\gamma}$ is unitary, which implies 
the orthogonality relations \eqref{eq: CG Ort 1} and \eqref{eq: CG Ort 2} for the Clebsch--Gordan coefficients. 
\end{proof}

As a consequence of \eqref{eq: CG Ort 1}, we also have  
\begin{align}\label{eq:base change}
    v_{m,\alpha}\otimes v_{n,\beta}=\sum_{k,\gamma} \overline{C^{k,\gamma}_{m,\alpha;n,\beta}}u_{k,\gamma}.
\end{align}

\subsection{Schur orthogonality relations}
The Schur orthogonality relations compute the inner products between  matrix entries of irreducible representations of $G$. 

\begin{lemma}[Theorem 6.3.3 and 6.3.4 of \cite{Far08}]\label{lem:Schur}

    \noindent (1) For $m\in\mathbb{Z}_{\geq 0}$ and $u,u',v,v'\in  V_m$, 
    $$\int_G \langle\pi_m(\mathrm{g})u,v\rangle\overline{ \langle\pi_m(\mathrm{g})u',v'\rangle} \ {\rm d}\mu(\mathrm{g})=\frac{1}{m+1} \langle u,u'\rangle \overline{\langle v,v'\rangle}.$$
    (2) For distinct $m,m'\in \mathbb{Z}_{\geq 0}$, $u,v\in  V_m$, $u',v'\in  V_{m'}$,
    $$\int_G \langle\pi_m(\mathrm{g})u,v\rangle\overline{ \langle\pi_{m'}(\mathrm{g})u',v'\rangle} \ {\rm d}\mu(\mathrm{g})=0.$$
\end{lemma}

\subsection{Eigenfunctions and their products}
The Peter--Weyl theorem provides an orthogonal decomposition of $L^2(G)$:
$$L^2(G)=\widehat{\bigoplus_{m\in\mathbb{Z}_{\geq 0}}} E_m,$$
where $E_m$ is the linear space spanned by matrix entries of the form $\langle\pi_m(\mathrm{g}) u,v\rangle$ with $u,v\in  V_m$. At the same time, this provides the spectral decomposition of the Laplace--Beltrami operator $\Delta_G$, and $E_m$ is exactly the space of eigenfunctions of $\Delta_G$ with eigenvalue $-m(m+2)$ (see Section 8.3 of \cite{Far08}). 
The following lemma is a direct consequence of 
(1) of Lemma \ref{lem:Schur}. 

\begin{lemma}\label{lem:eigenbasis}
    Let $\{v_{m,\alpha}: \alpha=-m,-m+2,\ldots,m\}$ be an orthonormal basis of $ V_m$ with respect to the $\pi_m$-invariant inner product $\langle\ ,\ \rangle$. Then 
    $$\{{\sqrt{m+1}}\langle \pi_m(\mathrm{g})v_{m,\alpha},v_{m,\alpha'}\rangle: \alpha,\alpha'\in\{-m,-m+2,\ldots,m\}\}$$
    is an orthonormal basis of $E_m$. 
\end{lemma}

For $m,n\in\mathbb{Z}_{\geq 0}$, let $f,g$ be eigenfunctions of $\Delta_G$ such that 
    $$\Delta_{G} f=-m(m+2)f,\ \ \ \Delta_{G} g=-n(n+2)g. $$
Using the above lemma, we may write 
\begin{align}\label{eq: f=sum}
    f(\mathrm{g})=\sum_{\alpha,\alpha'\in\{-m,-m+2,\ldots,m\}}a_{\alpha,\alpha'}{\sqrt{m+1}}\langle \pi_m(\mathrm{g})v_{m,\alpha},v_{m,\alpha'}\rangle,
\end{align}
\begin{align}\label{eq: g=sum}
g(\mathrm{g})=\sum_{\beta,\beta'\in\{-n,-n+2,\ldots,n\}}b_{\beta,\beta'}{\sqrt{n+1}}\langle \pi_{n}(\mathrm{g})v_{n,\beta},v_{n,\beta'}\rangle,
\end{align}
where $a_{\alpha,\alpha'}, b_{\beta,\beta'}\in \mathbb{C}$ such that 
\begin{align}\label{eq: f,g,l2}
    \|f\|_{L^2(G)}=\|a_{\alpha,\alpha'}\|_{\ell^2_{\alpha,\alpha'}}\text{ and }\|g\|_{L^2(G)}=\|b_{\beta,\beta'}\|_{\ell^2_{\beta,\beta'}}.
\end{align}
To prove Theorem \ref{mainthm: 1}, we need the following general form for the product $fg$.  

\begin{lemma}\label{lem: fg=sum}
With the notation of Theorem \ref{thm:CG}, and $f,g$ in \eqref{eq: f=sum}, \eqref{eq: g=sum}, we have 
\begin{align*}
    fg    &=(m+1)^{\frac12}(n+1)^{\frac12}\sum_k\sum_{\alpha,\alpha',\beta,\beta',\gamma,\gamma'}a_{\alpha,\alpha'}b_{\beta,\beta'} \overline{C^{k,\gamma}_{m,\alpha;n,\beta}}C^{k,\gamma'}_{m,\alpha';n,\beta'}\langle  \pi_k(\mathrm{g})(u_{k,\gamma}), u_{k,\gamma'}\rangle.
\end{align*}
\end{lemma}
\begin{proof}
Using tensor products, we write  
\begin{align*}
    \langle \pi_m(\mathrm{g})v_{m,\alpha},v_{m,\alpha'}\rangle\langle \pi_{n}(\mathrm{g})v_{n,\beta},v_{n,\beta'}\rangle
    &=\langle(\pi_m\otimes\pi_{n}) (\mathrm{g})(v_{m,\alpha}\otimes v_{n,\beta}),v_{m,\alpha'}\otimes v_{n,\beta'}\rangle.
\end{align*}

Applying equation \eqref{eq:base change}, we have 
$$ v_{m,\alpha}\otimes v_{n,\beta}=\sum_{k,\gamma} \overline{C^{k,\gamma}_{m,\alpha;n,\beta}}u_{k,\gamma}$$
and
$$
\ v_{m,\alpha'}\otimes v_{n,\beta'}=\sum_{k,\gamma'} \overline{C^{k,\gamma'}_{m,\alpha';n,\beta'}}u_{k,\gamma'}.$$
Applying (1) of Theorem \ref{thm:CG}, we obtain  
$$(\pi_m\otimes\pi_n)(\mathrm{g})(v_{m,\alpha}\otimes v_{n,\beta})=\sum_{k,\gamma}\overline{C^{k,\gamma}_{m,\alpha;n,\beta}} \pi_k(\mathrm{g})(u_{k,\gamma}),$$
where in the summation $k$ ranges over $m-n,m-n+2,\ldots,m+n$. 
Using the fact that $\langle u_{k,\gamma},u_{k',\gamma'}\rangle=0$ for distinct $k,k'$, we have 
\begin{align*}
    \langle(\pi_m\otimes\pi_n)(\mathrm{g})(v_{m,\alpha}\otimes v_{n,\beta}),v_{m,\alpha'}\otimes v_{n,\beta'}\rangle
    &=\l\langle\sum_{k,\gamma}\overline{C^{k,\gamma}_{m,\alpha;n,\beta}} \pi_k(\mathrm{g})(u_{k,\gamma}),\sum_{k,\gamma'} \overline{C^{k,\gamma'}_{m,\alpha';n,\beta'}}u_{k,\gamma'}\r\rangle\\
    &=\sum_{k,\gamma,\gamma'} \l\langle \overline{C^{k,\gamma}_{m,\alpha;n,\beta}} \pi_k(\mathrm{g})(u_{k,\gamma}), \overline{C^{k,\gamma'}_{m,\alpha';n,\beta'}}u_{k,\gamma'}\r\rangle\\
    &=\sum_{k,\gamma,\gamma'} \overline{C^{k,\gamma}_{m,\alpha;n,\beta}}C^{k,\gamma'}_{m,\alpha';n,\beta'}\langle  \pi_k(\mathrm{g})(u_{k,\gamma}), u_{k,\gamma'}\rangle.
\end{align*}
Thus, by \eqref{eq: f=sum} and \eqref{eq: g=sum}, we conclude that
\begin{align*}
    fg&=(m+1)^{\frac12}(n+1)^{\frac12}\sum_{\alpha,\alpha',\beta,\beta'}a_{\alpha,\alpha'}b_{\beta,\beta'}\sum_{k,\gamma,\gamma'} \overline{C^{k,\gamma}_{m,\alpha;n,\beta}}C^{k,\gamma'}_{m,\alpha';n,\beta'}\langle  \pi_k(\mathrm{g})(u_{k,\gamma}), u_{k,\gamma'}\rangle \\
    &=(m+1)^{\frac12}(n+1)^{\frac12}\sum_{k}\sum_{\alpha,\alpha',\beta,\beta',\gamma,\gamma'}a_{\alpha,\alpha'}b_{\beta,\beta'} \overline{C^{k,\gamma}_{m,\alpha;n,\beta}}C^{k,\gamma'}_{m,\alpha';n,\beta'}\langle  \pi_k(\mathrm{g})(u_{k,\gamma}), u_{k,\gamma'}\rangle.
\end{align*}
\end{proof}

As an immediate corollary, we have the standard 
\begin{lemma}\label{lem: fg G supp}
        For $m,n\in\mathbb{Z}_{\geq 0}$, let $f,g$ be eigenfunctions of $\Delta_{\bS^3}$ such that 
    $$\Delta_{\bS^3} f=-m(m+2)f \quad
    \text{and} \quad
   \Delta_{\bS^3} g=-n(n+2)g. $$
    Assume that $m\geq n$.
    Then the product $fg$ is a sum of eigenfunctions of $\Delta_{\bS^3}$ with eigenvalues $-k(k+2)$, where $k\in \{m-n,m-n+2,\ldots,m+n\}$. 
\end{lemma} 
\begin{proof}
    By Lemma \ref{lem: fg=sum}, we see that $fg$ is a sum of functions of the form $\langle  \pi_k(\mathrm{g})(u_{k,\gamma_1}), u_{k,\gamma_2}\rangle$, where $k$ ranges over $m-n,m-n+2,\ldots,m+n$, and $\{u_{k,\gamma}\}_\gamma$ is an orthonormal basis of the underlying vector space of the irreducible representation $\pi_k$. Since any $\langle  \pi_k(\mathrm{g})(u_{k,\gamma_1}), u_{k,\gamma_2}\rangle$ is an eigenfunction of $\Delta_G$ with eigenvalue $-k(k+2)$, the proof is complete.  
\end{proof}

\section{Sharp bilinear eigenfunction estimate on $\mathbb{S}^3$: Proof of Theorem \ref{mainthm: 1}}\label{sec: bilinear}

In this section, we prove Theorem \ref{mainthm: 1}. We identify $\bS^3$ with $\mathrm{SU}(2)$, so that we may apply Lemma \ref{lem: fg=sum} to express the product of eigenfunctions as a linear combination of matrix entries of irreducible representations. After applying the Schur orthogonality relations, to finish the proof it suffices to use the orthogonality relations for the Clebsch--Gordan coefficients detailed in Theorem \ref{thm:CG} and some elementary estimates.

\begin{proof}[Proof of Theorem \ref{mainthm: 1}]
We assume $m\geq 2n$; the case $n\leq m< 2n$ can be handled by Sogge's $L^4$ eigenfunction bound \cite{Sog88} combined with H\"older's inequality:
$$\|fg\|_{L^2}\leq \|f\|_{L^4}\|g\|_{L^4}\lesssim (m+1)^{\frac14}(n+1)^{\frac14}\|f\|_{L^2}\|g\|_{L^2}\lesssim (n+1)^{\frac12}\|f\|_{L^2}\|g\|_{L^2}. $$
We identify $\bS^3$ with the group $\mathrm{SU}(2)$. 
With the notation of Theorem \ref{thm:CG}, we write $f,g$ as in \eqref{eq: f=sum}, \eqref{eq: g=sum}, with \eqref{eq: f,g,l2}. 
It suffices to prove 
$$\|fg\|_{L^2(G)}\lesssim (n+1)^{\frac12}\|a_{\alpha,\alpha'}\|_{l^2_{\alpha,\alpha'}}\|b_{\beta,\beta'}\|_{l^2_{\beta,\beta'}}.$$
Then Lemma \ref{lem: fg=sum} provides an explicit expression for the product $fg$. Using this and applying (2) of Lemma \ref{lem:Schur}, we obtain 
\begin{align*}
    \|fg\|_{L^2(G)}^2&=(m+1)(n+1)\sum_k \l\|\sum_{\alpha,\alpha',\beta,\beta',\gamma,\gamma'}a_{\alpha,\alpha'}b_{\beta,\beta'}\overline{C^{k,\gamma}_{m,\alpha;n,\beta}}C^{k,\gamma'}_{m,\alpha';n,\beta'}\langle  \pi_k(\mathrm{g})(u_{k,\gamma}), u_{k,\gamma'}\rangle\r\|^2_{L^2(G)}\\
    &=(m+1)(n+1)\sum_k\sum_{\substack{\alpha,\alpha',\beta,\beta',\gamma,\gamma'\\ \tilde{\alpha},\tilde{\alpha}',\tilde{\beta},\tilde{\beta}',\tilde{\gamma},\tilde{\gamma}'}}
    a_{\alpha,\alpha'}b_{\beta,\beta'}\overline{C^{k,\gamma}_{m,\alpha;n,\beta}}C^{k,\gamma'}_{m,\alpha';n,\beta'}\overline{ a_{\tilde{\alpha},\tilde{\alpha}'}b_{\tilde{\beta},\tilde{\beta}'} }C^{k,\tilde{\gamma}}_{m,\tilde{\alpha};n,\tilde{\beta}}\overline{C^{k,\tilde{\gamma}'}_{m,\tilde{\alpha}';n,\tilde{\beta}'}}\\
    &\ \ \ \cdot \int_G
    \langle  \pi_k(\mathrm{g})(u_{k,\gamma}), u_{k,\gamma'}\rangle\overline{\langle  \pi_k(\mathrm{g})(u_{k,\tilde{\gamma}}), u_{k,\tilde{\gamma}'}\rangle}
    \ {\rm d}\mu(\mathrm{g}).
\end{align*}
Recall from Theorem \ref{thm:CG} that $\{u_{k,\gamma}: \gamma=-k,-k+2,\ldots,k\}$ is an orthonormal family of $ V_m\otimes V_n$, then we may apply (1) of Lemma \ref{lem:Schur}, to obtain 
\begin{align*}
    \|fg\|_{L^2(G)}^2
     &=(m+1)(n+1)\sum_k\sum_{\substack{\alpha,\alpha',\beta,\beta',\gamma,\gamma'\\ \tilde{\alpha},\tilde{\alpha}',\tilde{\beta},\tilde{\beta}',\tilde{\gamma},\tilde{\gamma}'}}
    a_{\alpha,\alpha'}b_{\beta,\beta'}\overline{C^{k,\gamma}_{m,\alpha;n,\beta}}C^{k,\gamma'}_{m,\alpha';n,\beta'}\overline{ a_{\tilde{\alpha},\tilde{\alpha}'}b_{\tilde{\beta},\tilde{\beta}'} }C^{k,\tilde{\gamma}}_{m,\tilde{\alpha};n,\tilde{\beta}}\overline{C^{k,\tilde{\gamma}'}_{m,\tilde{\alpha}';n,\tilde{\beta}'}}\\
    &\ \ \ \cdot 
    \frac{1}{k+1}\langle u_{k,\gamma}, u_{k,\tilde{\gamma}}\rangle \overline{ \langle u_{k,\gamma'}, u_{k,\tilde{\gamma}'}\rangle}\\
    &=(m+1)(n+1)\sum_k\sum_{\substack{\alpha,\alpha',\beta,\beta',\gamma,\gamma'\\ \tilde{\alpha},\tilde{\alpha}',\tilde{\beta},\tilde{\beta}',\tilde{\gamma},\tilde{\gamma}'}}
    a_{\alpha,\alpha'}b_{\beta,\beta'}\overline{C^{k,\gamma}_{m,\alpha;n,\beta}}C^{k,\gamma'}_{m,\alpha';n,\beta'}\overline{ a_{\tilde{\alpha},\tilde{\alpha}'}b_{\tilde{\beta},\tilde{\beta}'} }C^{k,\tilde{\gamma}}_{m,\tilde{\alpha};n,\tilde{\beta}}\overline{C^{k,\tilde{\gamma}'}_{m,\tilde{\alpha}';n,\tilde{\beta}'}}\\
    &\ \ \ \cdot 
    \frac{1}{k+1}
    \delta_{\gamma,\tilde{\gamma}}\delta_{\gamma',\tilde{\gamma}'}\\
    &=(n+1)\sum_k \frac{m+1}{k+1}\sum_{\substack{\alpha,\alpha',\beta,\beta',\gamma,\gamma'\\ \tilde{\alpha},\tilde{\alpha}',\tilde{\beta},\tilde{\beta}'}}
    a_{\alpha,\alpha'}b_{\beta,\beta'}\overline{C^{k,\gamma}_{m,\alpha;n,\beta}}C^{k,\gamma'}_{m,\alpha';n,\beta'}\overline{ a_{\tilde{\alpha},\tilde{\alpha}'}b_{\tilde{\beta},\tilde{\beta}'} }C^{k,\gamma}_{m,\tilde{\alpha};n,\tilde{\beta}}\overline{C^{k,\gamma'}_{m,\tilde{\alpha}';n,\tilde{\beta}'}}.
\end{align*}

Now for $\gamma,\gamma'\in \{-m-n,-m-n+2,\ldots,m+n\}$, let 
$$S(k,\gamma,\gamma')=\sum_{\alpha,\alpha',\beta,\beta'}a_{\alpha,\alpha'}b_{\beta,\beta'}\overline{C^{k,\gamma}_{m,\alpha;n,\beta}
}C^{k,\gamma'}_{m,\alpha';n,\beta'}.
$$
Then 
$$\|fg\|_{L^2(G)}^2=(n+1)\sum_{k,\gamma,\gamma'}\frac{m+1}{k+1}\l|S(k,\gamma,\gamma')\r|^2.$$
Recall that $k\in\{m-n,m-n+2,\ldots,m+n\}$, so under the  assumption that $m\geq 2n$, we have crucially that
\begin{align}\label{eq: crucial}
    \frac{m+1}{k+1}\sim 1.
\end{align}
Thus, it suffices to prove 
\begin{align}\label{eq: |S|^2}
  \sum_{k,\gamma,\gamma'}|S(k,\gamma,\gamma')|^2\leq \|a_{\alpha,\alpha'}\|_{\ell^2_{\alpha,\alpha'}}^2\|b_{\beta,\beta'}\|_{\ell^2_{\beta,\beta'}}^2
  =\|a_{\alpha,\alpha'}b_{\beta,\beta'}\|^2_{\ell^2_{\alpha,\alpha',\beta,\beta'}}.
\end{align}
In fact, the above follows as a Bessel's inequality in $\ell^2_{\alpha,\alpha',\beta,\beta'}$---it suffices to check that the vectors 
$$v^{k,\gamma,\gamma'}:=C^{k,\gamma}_{m,\alpha;n,\beta}
\overline{C^{k,\gamma'}_{m,\alpha';n,\beta'}}$$
form an orthonormal family in $\ell^2_{\alpha,\alpha',\beta,\beta'}$. For this purpose, we check by the orthogonality relations \eqref{eq: CG Ort 2} for the Clebsch--Gordan coefficients, that 
\begin{align*}
\langle v^{k,\gamma,\gamma'}, v^{\tilde{k},\tilde{\gamma},\tilde{\gamma}'}\rangle_{\ell^2_{\alpha,\alpha',\beta,\beta'}}
&=\sum_{\alpha,\alpha',\beta,\beta'}C^{k,\gamma}_{m,\alpha;n,\beta}
\overline{C^{k,\gamma'}_{m,\alpha';n,\beta'}}\overline{C^{\tilde{k},\tilde{\gamma}}_{m,\alpha;n,\beta}
}C^{\tilde{k},\tilde{\gamma}'}_{m,\alpha';n,\beta'}\\
&=\l(\sum_{\alpha,\beta}
C^{k,\gamma}_{m,\alpha;n,\beta}\overline{C^{\tilde{k},\tilde{\gamma}}_{m,\alpha;n,\beta}}\r) 
\l(\sum_{\alpha',\beta'}C^{\tilde{k},\tilde{\gamma}'}_{m,\alpha';n,\beta'}
\overline{C^{k,\gamma'}_{m,\alpha';n,\beta'}}\r)
\\
&=\delta_{k,\tilde{k}} \delta_{\gamma,\tilde{\gamma}}\cdot \delta_{\tilde{k},k} \delta_{\tilde{\gamma}',\gamma'}
=\delta_{k,\tilde{k}} \delta_{\gamma,\tilde{\gamma}}\delta_{\gamma',\tilde{\gamma}'}. 
\end{align*}
The proof is completed. 

\end{proof}

\begin{remark}\label{rem: Sogge?}
Adapting our argument to Sogge’s $L^4$ eigenfunction estimate on $\bS^3$ appears challenging. While our approach to the bilinear estimate avoids oscillatory integral techniques, Sogge’s linear estimate relies on them, indicating that the two problems call for different perspectives.
\end{remark}

\begin{remark}\label{generalization}
By identifying the standard sphere $\mathbb{S}^d$ with the symmetric space
$\mathrm{SO}(d+1)/\mathrm{SO}(d)$ and exploiting the associated spherical
representations of the special orthogonal group, the method developed in this
section extends naturally to spheres of arbitrary dimension, and can 
yield sharp bilinear eigenfunction estimates on $\mathbb{S}^d$, thereby
recovering the results of \cite{BGT05,BGT052} in the spherical setting.
Let the \emph{spherical} Clebsch--Gordan coefficients $C_{m,n}^k$ be defined by
the decomposition
\[
e_m \otimes e_n = \sum_k C_{m,n}^k\, e_k,
\]
where $e_m,e_n,e_k$ denote the unit spherical vectors in the corresponding
spherical representations of $\mathrm{SO}(d+1)$ indexed by the nonnegative integers $m,n,k$ respectively.
The desired bilinear eigenfunction estimate can be reduced to establishing suitable
bounds for these coefficients. Specifically, in the high-low regime
$m \ge 2n$, we claim that
\[
\sup_k |C_{m,n}^k| \;\lesssim\;
\begin{cases}
(n+1)^{-1/4}, & d=2, \\ 
(n+1)^{-1/2}, & d \ge 3.
\end{cases}
\]
In particular, the case $d=3$ follows directly from the orthogonality relations
for general Clebsch--Gordan coefficients, and may be viewed as a reformulation of the
argument in the proof of Theorem~\ref{mainthm: 1}. The general cases can be treated by expressing $C_{m,n}^k$ in terms of
Gegenbauer linearization coefficients, for which explicit formulas are
available (see \cite{AAR99}), allowing for a direct analysis
across dimensions. 
Overall, our approach provides a purely algebraic and structurally transparent alternative to the classical microlocal techniques.  
\end{remark}

\begin{corollary}[Sharp multilinear eigenfunction estimate]\label{cor: Trilinear} Let $k\in\mathbb{Z}_{\geq 2}$, and let $m_i\in\mathbb{Z}_{\geq 0}$, $i=1,2,\ldots,k$. Assume that 
$m_1\geq m_2\geq\cdots\geq m_k$. Let $f_i$ be an eigenfunction of $\Delta_{\bS^3}$ such that 
    $$\Delta_{\bS^3} f_i=-m_i(m_i+2)f_i, \ i=1,\ldots,k.$$    
    Then 
    $$\l\|\prod_{i=1}^kf_i\r\|_{L^2(\bS^3)}\lesssim \l((m_2+1)^{\frac12}\prod_{i=3}^k(m_i+1)\r)\prod_{i=1}^k\|f_i\|_{L^2(\bS^3)}.$$       
\end{corollary}
\begin{proof}   
As observed in \cite{BGT052}, it suffices to apply the classical bound valid on any compact manifold \cite{Ava56, Lev52} to all $f_i$ with $i \geq 3$,
\[
\|f_i\|_{L^\infty(\bS^3)} \lesssim (m_i+1)\|f_i\|_{L^2(\bS^3)},
\]
while handling the product $f_1 f_2$ using Theorem \ref{mainthm: 1}.
\end{proof}

\begin{remark}
    As shown in \cite{BGT052}, both Theorem \ref{mainthm: 1} and Corollary \ref{cor: Trilinear} are sharp, as can be seen by testing against zonal spherical harmonics. 
\end{remark}

\section{Refined anisotropic Strichartz estimate on $\bR\times\bT$: Proof of Theorem \ref{mainthm: 2}}\label{sec:  Strichartz}

In this section, we give the proof of Theorem \ref{mainthm: 2}. Before giving the proof, we first present the required measure estimate on $\mathbb{R} \times \mathbb{Z}$ and the counting estimate on $\mathbb{Z}^2$.

\begin{lemma}[Lemma 2.1 of \cite{TT01}, or Lemma 3.1 of \cite{HTT14}]    \label{lem:measure estimate lem}
Let $K\ge 1$. Then
$$ \sup_{C\in \bR,\ \xi'\in \bR \times \bZ} \l|\l\{ \xi\in \bR \times \bZ: C\le |\xi-\xi'|^2 \le C+K   \r\}    \r|\lea K, $$
where the outer $|\cdot|$ denotes the standard measure on $\R\times\bZ$, which is the product of the one-dimensional Lebesgue measure on $\R$ and the counting measure on $\bZ$.

\end{lemma}

\begin{lemma}\label{lem:counting lem}
Let $N \in \bZ_{\geq 1}.$ Then for any $\varepsilon>0$, the following hold:
\begin{equation}\label{eq:counting lem-1}
 \sup_{k\in \bZ, C\in \bZ } \l|  \{(m,n)\in \bZ^2: |m|, |n|\le N, m^2+n^2+km+kn=C       \}                   \r| \lea_{\varepsilon} N^\varepsilon,
\end{equation}
and
\begin{equation}\label{eq:counting lem-2}
\sup_{k\in \bZ, C\in \bZ } \l|  \{(m,n)\in \bZ^2: m\ne 0, n\ne 0, |m-k|\lea N, |n|\lea N, mn= C       \}                   \r| \lea_{\varepsilon} N^\varepsilon.
\end{equation}
\end{lemma}
\begin{proof}[Proof of Lemma \ref{lem:counting lem}]
We first prove \eqref{eq:counting lem-1}. 
The equation in \eqref{eq:counting lem-1} implies 
$$     (2m+k)^2+(2n+k)^2=4C+2k^2.               $$
If $|k|\lea N^{10}$, then $(2m+k)^2+(2n+k)^2\lesssim N^{20}$ since $|m|,|n|\leq N$, then \eqref{eq:counting lem-1} follows from the standard arithmetic result that the number of lattice points on the circle $x^2+y^2=K$ is $O(K^\varepsilon)$.
Hence, we may assume $|k|\gg N^{10}$. For any two points $(m_1,n_1), (m_2, n_2)$ satisfying
$$ m_1^2+n_1^2+km_1+kn_1=m_2^2+n_2^2+km_2+kn_2=C,$$
we have 
$$ |k(m_1+n_1-m_2-n_2)|=|m_1^2+n_1^2- m_2^2-n_2^2  |\lea N^2 \ll |k|,$$
which implies $m_1+n_1-m_2-n_2=0$, and thus $m_1^2+n_1^2- m_2^2-n_2^2=0$.  
Since the intersection of a line and a circle consists of at most two points, we have 
$$ \sup_{k\in \bZ, |k|\gg N^{10}, C\in \bZ } \l|  \{(m,n)\in \bZ^2: |m|, |n|\le N, m^2+n^2+km+kn=C       \}                    \r|\le 2.$$

We now prove \eqref{eq:counting lem-2}. If $|C|\lea N^{10}$, then \eqref{eq:counting lem-2} holds by the divisor bound. On the other hand if $|C|\gg N^{10}$, then for any two points $(m_1,n_1), (m_2, n_2)$ satisfying
$$ m_1 n_1= m_2 n_2=C,$$
we have 
$$ |C(n_1-n_2)|=|n_1 n_2(m_1-m_2) |\lea N^3 \ll |C|,$$
which implies $n_1-n_2=m_1-m_2=0$.  This implies
$$ \sup_{k\in \bZ, C\in \bZ, |C|\gg N^{10} } \l|  \{(m,n)\in \bZ^2: m\ne 0, n\ne 0, |m-k|\lea N, |n|\lea N, mn= C       \}                   \r| \le 1.$$
This completes the proof.
\end{proof}

We are now ready to present the proof of Theorem \ref{mainthm: 2}.  
The overarching strategy is to unfold the quartic $L^4$ functional to expose its multilinear structure, as is likewise done in \cite{DGG17, HK24}. 
Additionally, following \cite{HTT14}, we split the argument into cases based on the direction of the vector $a$.

\begin{proof}[Proof of Theorem \ref{mainthm: 2}]
Without loss of generality, we may assume that $a\cdot\xi^0-c=0$. This is because we may enlarge $N$ to $N'$ which is at most $3N$, and $M$ to $M'$ which is at most $3M$,  such that 
\begin{align*}
&\{\xi\in\R\times\bZ: |\xi-\xi^0|\leq N, |a\cdot\xi-c|\leq M\}\\
\subset& \{\xi\in\R\times\bZ: |\xi-\widetilde{\xi^0}|\leq N', |a\cdot(\xi-\widetilde{\xi^0})|\leq M'\},
\end{align*}
where $\widetilde{\xi^0}$ is some element in the set $\{\xi\in\R\times\bZ: |\xi-\xi^0|\leq N, |a\cdot\xi-c|\leq M\}$.

Next, we apply the Galilean transform, which amounts to the change of variables $\xi\mapsto\xi-\w{\xi^0}=(\xi_1-\w{\omega^0},\xi_2-\w{k^0})$, to obtain
\begin{align*}
&\l\|\varphi(t)\int_{\bR \times \bZ} e^{ix_1\cdot \xi_1-it|\xi|^2} \widehat{\phi}(\xi) \ {\rm d}\xi\r\|_{L^4_{t,x_1}(\bR\times \bR)}\\
=&\l\|\varphi(t)e^{ix_1\cdot \w{\omega^0}-it|\w{\xi^0}|^2}\int_{\bR \times \bZ} e^{i(x_1-2t\w{\omega^0})\cdot \xi_1-it(|\xi|^2+2\w{k^0}\xi_2)} \widehat{\phi}(\xi+\w{\xi^0})\ {\rm d}\xi\r\|_{L^4_{t,x_1}(\bR\times \bR)}\\
=&\l\|\varphi(t)\int_{\bR \times \bZ} e^{ix_1\cdot \xi_1-it(|\xi|^2+2\w{k^0}\xi_2)} \widehat{\phi}(\xi+\w{\xi^0}) \ {\rm d}\xi\r\|_{L^4_{t,x_1}(\bR\times \bR)},
\end{align*}
where in the last equality we used the translation invariance of the $L^4_{x_1}(\R)$ norm. 
Thus it suffices to prove
$$\l\|\varphi(t)\int_{\bR \times \bZ} e^{ix_1\cdot \xi_1-it(|\xi|^2+k\xi_2)} \widehat{\phi}(\xi) \ {\rm d}\xi\r\|_{L^4_{t,x_1}(\bR\times \bR)} \lea \l(\frac{M}{N}\r)^{\delta} N^{\frac14}   \|\phi\|_{L^2},$$
where 
$$\text{supp}(\widehat{\phi})\subset \mathcal{R}=\{\xi=(\xi_1,\xi_2)\in \bR\times \bZ : |\xi|\le N,\ |a\cdot \xi|\le M  \},$$
uniformly in the parameters $a\in\bR^2$ with $|a|=1$, $c\in \R$, and $k\in\bZ$. 
Without loss of generality, we may also assume that $k\geq 0$ and $\|\phi\|_{L^2}=1$. 

We introduce the following notation: 
$$\vec{\xi}:=(\xi^{(1)},\xi^{(2)},\xi^{(3)},\xi^{(4)})\in (\R\times \bZ)^4,$$
$$\ {\rm d}\vec{\xi}:=\ {\rm d}\xi^{(1)}\ {\rm d}\xi^{(2)}\ {\rm d}\xi^{(3)}\ {\rm d}\xi^{(4)},$$
$$\widehat{\phi}(\vec{\xi}):=  \widehat{\phi}(\xi^{(1)}) \widehat{\phi}(\xi^{(3)}) \overline{\widehat{\phi}(\xi^{(2)}) \widehat{\phi}(\xi^{(4)}) },$$
and 
$$\langle F(\xi)\rangle:=F(\xi^{(1)})+F(\xi^{(3)})-F(\xi^{(2)})-F(\xi^{(4)}).$$
We now estimate 
\begin{align}
\l\|\varphi(t)\int_{\bR \times \bZ} e^{ix_1\cdot \xi_1-it(|\xi|^2+k \xi_2)} \widehat{\phi}(\xi) \ {\rm d} \xi \r\|_{L^4_{t,x_1}(\bR\times \bR)}^4 
\lea &    \l\| \varphi(t)^{\frac14}    \int_{\bR \times \bZ} e^{ix_1\cdot \xi_1-it(|\xi|^2+k \xi_2)} \widehat{\phi}(\xi) \ {\rm d} \xi  \r\|_{L^4_{t,x_1}(\bR\times \bR)}^4  \nonumber          \\
=&(2\pi)^2 \int_{\mathcal{R}^4} \delta_0(\langle\xi_1\rangle)
\widehat{\varphi}(\langle|\xi|^2 +k \xi_2\rangle)\widehat{\phi}(\vec{\xi}) \ {\rm d}\vec{\xi} \nonumber          \\
\lea & \int_{\mathcal{R}^4} \delta_0(\langle \xi_1 \rangle) \ca_{|\langle |\xi|^2 +k \xi_2 \rangle|\lea 1} |\widehat{\phi}(\vec{\xi})| \ {\rm d}\vec{\xi}. \label{eq:4linear}
\end{align}
We then split our argument according to the direction of the unit vector $a=(a_1,a_2)\in\R^2$.

{\bf Case 1.} $|a_2|\gea \l(\frac{M}{N} \r)^{1-{4\delta}}$. This case is readily tractable, as the directional constraint effectively yields a manageable restriction along an integer direction, enabling us to bound \eqref{eq:4linear} via measure estimates on $\mathbb{R} \times \mathbb{Z}$. We use
$$ |\widehat{\phi}(\vec{\xi})|\lea |\widehat{\phi}(\xi^{(1)}) \widehat{\phi}(\xi^{(3)})  |^2  +|\widehat{\phi}(\xi^{(2)}) \widehat{\phi}(\xi^{(4)})  |^2.$$
Then by symmetry in the variables $\xi^{(j)}$, it suffices to show
$$   \int_{\mathcal{R}^4} \delta_0(\langle \xi_1 \rangle) \ca_{|\langle |\xi|^2 +k \xi_2 \rangle|\lea 1} |\widehat{\phi}(\xi^{(2)}) \widehat{\phi}(\xi^{(4)})  |^2   \ {\rm d}\vec{\xi} \lea   \l(\frac{M}{N}\r)^{4\delta} N.   $$
This follows from
$$ \sup_{\xi^{(2)}, \xi^{(4)} \in \mathcal{R}}  \int_{\mathcal{R}^2} \delta_0(\langle \xi_1 \rangle) \ca_{|\langle |\xi|^2 +k \xi_2 \rangle|\lea 1}  \ {\rm d}\xi^{(1)} \ {\rm d}\xi^{(3)}   \lea \l(\frac{M}{N}\r)^{4\delta} N.   $$
Fix $\xi^{(2)}, \xi^{(4)} \in \mathcal{R}$. We make the linear change of variables 
\begin{align*}
    \left\{
    \begin{array}{lll}
        v&=&\xi_1^{(1)}- \xi_1^{(3)} , \\
        m&=&\xi_2^{(1)}+ \xi_2^{(3)},\\
        n&=&\xi_2^{(1)}- \xi_2^{(3)}. 
    \end{array}
    \right.
\end{align*}
The factor $\delta_0(\langle\xi_1\rangle)$ of the integrand gives
$$\xi_1^{(1)}+\xi_1^{(3)}=\xi_1^{(2)}+\xi_1^{(4)},$$
which implies 
$$a\cdot \langle\xi\rangle=a_2\langle\xi_2\rangle=a_2(m-\xi_2^{(2)}-\xi_2^{(4)}),$$
and 
$$(\xi_1^{(1)})^2+(\xi_1^{(3)})^2=\frac12 v^2+\frac{1}{2}(\xi_1^{(2)}+\xi_1^{(4)})^2,$$
so that 
$$\langle|\xi|^2+k\xi_2\rangle=\frac12 v^2+\frac{1}{2}(\xi_1^{(2)}+\xi_1^{(4)})^2+\frac{1}{2}(m^2+n^2)-(\xi^{(2)})^2-(\xi^{(4)})^2+km-k(\xi_2^{(2)}+\xi_2^{(4)}).$$

Under the condition that 
$$\xi^{(j)}\in \mathcal{R},\ 1\le j \le 4,$$
we have  
$$| a_2(m-\xi_2^{(2)}-\xi_2^{(4)})|\leq 4 M,$$
which implies, under our assumption on $a_2$, that 
$$|m-\xi_2^{(2)}-\xi_2^{(4)}|\lea M^{4\delta} N^{1-{4\delta}}.$$  
Thus, the integral is bounded by
\begin{align*}
   & \int_{\mathcal{R}^2} \delta_0(\langle \xi_1 \rangle) \ca_{|\langle |\xi|^2 +k \xi_2 \rangle|\lea 1}\ {\rm d}\xi^{(1)} \ {\rm d}\xi^{(3)}\\
\lea & \sup_{C\in \bR}\l|  \l\{(v, m, n)\in \bR\times \bZ\times \bZ: |m-\xi_2^{(2)}-\xi_2^{(4)}|\lea M^{4\delta} N^{1-{4\delta}}, \l|\frac12 v^2+ \frac12(m^2+n^2)+km -C \r|\lea 1                \r\}      \r|\\
\lea & M^{4\delta} N^{1-{4\delta}} \sup_{C\in \bR} \l|  \l\{(v, n)\in \bR\times \bZ:  \l|\frac12 v^2+ \frac12 n^2 -C \r|\lea 1                \r\}      \r| \lea  \l( \frac{M}{N}\r)^{4\delta} N,
\end{align*}
where for the last inequality, we used Lemma \ref{lem:measure estimate lem}. 

{\bf Case 2}: $|a_2|\ll \l(\frac{M}{N} \r)^{1-{4\delta}}$. 
In this case we have $|a_1|\gea 1$, and thus
$$  \mathcal{R}\subset \mathcal{A}= \{\xi=(\xi_1,\xi_2)\in \bR\times \bZ : |\xi_1|\lea M^{1-{4\delta}} N^{4\delta},\ |\xi_2|\leq N  \}.$$ 
It suffices to obtain the desired estimate for the right-hand side of \eqref{eq:4linear} with $\mathcal{R}$ replaced by $\mathcal{A}$. 

By symmetry in the variables $\xi^{(j)}$, it suffices to prove 
$$\int_{\mathcal{A}^4} \delta_0(\langle \xi_1 \rangle) \ca_{\Gamma}(\vec{\xi})     |\widehat{\phi}(\vec{\xi})| \ {\rm d}\vec{\xi} \lea \l( \frac{M}{N}\r)^{4\delta} N, $$
where
$$ \Gamma=\{ \vec{\xi}\in \mathcal{A}^4: \xi_1^{(1)} \ge   \xi_1^{(3)}, \xi_1^{(2)} \ge   \xi_1^{(4)},  |\langle |\xi|^2 +k \xi_2 \rangle|\lea 1 \}.$$

Next, we \textit{decompose the kernel function} $\ca_{\Gamma}(\vec{\xi})$ into two components, bounding it from above by $K_1(\vec{\xi}) + K_2(\vec{\xi})$. We then control the quadrilinear form $|\widehat{\phi}(\vec{\xi})|$ using different arithmetic--geometric mean inequalities, and convert the resulting estimates into different measure estimates, analogous to Case 1. Define
\begin{align} \label{eq: K1}
K_1(\vec{\xi})=\ca_{\Gamma}(\vec{\xi}) \l(  \ca_{\xi_2^{(1)}= \xi_2^{(4)}}+ \ca_{\xi_2^{(3)}= \xi_2^{(2)}}+ \ca_{\xi_2^{(1)}+ \xi_2^{(4)}+k=0} +\ca_{\xi_2^{(3)}+ \xi_2^{(2)}+k=0}        \r) ,
\end{align}
and
$$    K_2(\vec{\xi})=\ca_{\Gamma}(\vec{\xi})   \ca_{\xi_2^{(1)}\ne \xi_2^{(4)}} \ca_{\xi_2^{(3)}\ne \xi_2^{(2)}} \ca_{\xi_2^{(1)} +\xi_2^{(4)}+k\ne 0} \ca_{\xi_2^{(3)}+ \xi_2^{(2)}+k\ne 0}        .            $$
Since 
\begin{align}\label{eq: K1 + K2}
    \ca_{\Gamma}(\vec{\xi}) \le  K_1(\vec{\xi})+K_2(\vec{\xi}),
\end{align}
it suffices  to prove
\begin{equation}\label{eq:main-case2-K_1}
\int_{\mathcal{A}^4} \delta_0(\langle \xi_1 \rangle) K_1(\vec{\xi})     |\widehat{\phi}(\vec{\xi})| \ {\rm d}\vec{\xi} \lea \l( \frac{M}{N}\r)^{4\delta} N,
\end{equation}
and
\begin{equation}\label{eq:main-case2-K_2}
  \int_{\mathcal{A}^4} \delta_0(\langle \xi_1 \rangle) K_2(\vec{\xi})     |\widehat{\phi}(\vec{\xi})| \ {\rm d}\vec{\xi} \lea \l( \frac{M}{N}\r)^{4\delta} N.
\end{equation}
For \eqref{eq:main-case2-K_1}, we use 
\begin{align} \label{eq: AM-GM 1}
|\widehat{\phi}(\vec{\xi})|\lea |\widehat{\phi}(\xi^{(1)}) \widehat{\phi}(\xi^{(3)})  |^2  +|\widehat{\phi}(\xi^{(2)}) \widehat{\phi}(\xi^{(4)})  |^2.
\end{align}
By symmetry in the variables $\xi^{(j)}$, to prove \eqref{eq:main-case2-K_1}, it suffices to show
$$     \int_{\mathcal{A}^4} \delta_0(\langle \xi_1 \rangle) K_1(\vec{\xi})   |\widehat{\phi}(\xi^{(2)}) \widehat{\phi}(\xi^{(4)})  |^2    \ {\rm d}\vec{\xi} \lea \l( \frac{M}{N}\r)^{4\delta} N.  $$
This follows from 
$$  \sup_{\xi^{(2)}, \xi^{(4)} \in \mathcal{A}}  \int_{\mathcal{A}^2} \delta_0(\langle \xi_1 \rangle)  K_1(\vec{\xi})  \ {\rm d}\xi^{(1)} \ {\rm d}\xi^{(3)}   \lea \l(\frac{M}{N}\r)^{4\delta} N. $$
Recall that $K_1(\vec{\xi})$ defined in \eqref{eq: K1} is a sum of four terms, and by symmetry it suffices to address the first and the third terms. Define $b=\xi_1^{(2)}+ \xi_1^{(4)}$.  
We need to prove
$$   \sup_{\xi^{(2)}, \xi^{(4)} \in \mathcal{A}}   \int_{\mathcal{A}^2}  \delta_0(\langle \xi_1 \rangle)  \ca_{\Gamma}(\vec{\xi})   \ca_{\xi_2^{(1)}= \xi_2^{(4)}}  \ {\rm d}\xi^{(1)} \ {\rm d}\xi^{(3)} \lea \l(\frac{M}{N}\r)^{4\delta} N,$$
and 
$$   \sup_{\xi^{(2)}, \xi^{(4)} \in \mathcal{A}}   \int_{\mathcal{A}^2}  \delta_0(\langle \xi_1 \rangle)  \ca_{\Gamma}(\vec{\xi})   \ca_{\xi_2^{(1)}+ \xi_2^{(4)}=k}  \ {\rm d}\xi^{(1)} \ {\rm d}\xi^{(3)} \lea \l(\frac{M}{N}\r)^{4\delta} N.  $$
The left-hand sides of the above two expression are both bounded by 
$$\sup_{b,C,k\in \bR} \l| \l\{ (\xi_1^{(1)}, \xi_2^{(3)} )\in \bR \times \bZ: \l||\xi_1^{(1)}|^2+ |b-\xi_1^{(1)}|^2 + |\xi_2^{(3)}|^2+k\xi_2^{(3)}+   C\r|\lea 1   \r\} \r| \lea 1,$$
where in the last inequality we used Lemma \ref{lem:measure estimate lem}. Observing that $ 1 \leq   \l( \frac{M}{N}\r)^{4\delta} N$, we complete the proof of \eqref{eq:main-case2-K_1}.

For \eqref{eq:main-case2-K_2}, we use 
\begin{align}\label{eq: AM-GM 2}
|\widehat{\phi}(\vec{\xi})|\lea |\widehat{\phi}(\xi^{(1)}) \widehat{\phi}(\xi^{(4)})  |^2  +|\widehat{\phi}(\xi^{(3)}) \widehat{\phi}(\xi^{(2)})  |^2 . 
\end{align}
By symmetry in the variables $\xi^{(j)}$, to prove \eqref{eq:main-case2-K_2}, it suffices to show
$$     \int_{\mathcal{A}^4} \delta_0(\langle \xi_1 \rangle) K_2(\vec{\xi}) |\widehat{\phi}(\xi^{(1)}) \widehat{\phi}(\xi^{(4)})  |^2        \ {\rm d}\vec{\xi} \lea \l( \frac{M}{N}\r)^{4\delta} N.               $$
This follows from 
\begin{equation}\label{eq:estimate K_2}
\sup_{\xi^{(1)}, \xi^{(4)} \in \mathcal{A}}  \int_{\mathcal{A}^2} \delta_0(\langle \xi_1 \rangle)  K_2(\vec{\xi})  \ {\rm d}\xi^{(2)} \ {\rm d}\xi^{(3)}   \lea \l(\frac{M}{N}\r)^{4\delta} N.
\end{equation}
Due to the factor $\delta_0(\langle \xi_1 \rangle)$, we may assume 
$\xi_1^{(1)}+\xi_1^{(3)}=\xi_1^{(2)}+\xi_1^{(4)}$. 
We make the affine change of variables 
\begin{align*}
    \left\{
    \begin{array}{lll}
        x&=&  \xi_1^{(1)}-\xi_1^{(3)}-\xi_1^{(2)}+\xi_1^{(4)}, \\
       m&=&\xi_2^{(3)}+\xi_2^{(2)}+k,     \\
        n&=& \xi_2^{(3)}-\xi_2^{(2)}.
    \end{array}\right.
\end{align*}
Define $l= \xi_1^{(1)}- \xi_1^{(4)} $ and $C= (\xi_2^{(1)})^2-(\xi_2^{(4)})^2+k(\xi_2^{(1)}-\xi_2^{(4)} )  $. In particular, $\xi_1^{(2)}- \xi_1^{(3)}=l$. 
Observe that under the condition $\vec{\xi}\in \Gamma$, 
we have $$|x|\leq 2l,$$
as well as 
$$0 \le l\lea M^{1-{4\delta}} N^{{4\delta}}.$$
Also observe that 
$$\langle|\xi|^2+k\xi_2\rangle=l x + mn  +C.$$
Now \eqref{eq:estimate K_2} reduces to the measure estimate 
\begin{align*}
   \l| \l\{ (x,m,n)\in \bR\times \bZ \times \bZ: |x|\le 2l, m\ne 0, n \ne 0, |m-k|\lea N, |n|\lea N,  \l| l x + mn  +C  \r|\lea 1         \r\}        \r| 
   &\lea   \l( \frac{M}{N}\r)^{4\delta} N. 
\end{align*} 
We prove this estimate in Lemma \ref{lem: main counting} as follows. This completes the proof of Theorem \ref{mainthm: 2}.
\end{proof}

\begin{lemma}\label{lem: main counting}
Fix $\delta\in (0,\frac18)$. 
Define 
\begin{align*}
   \mathcal{B}:&= \l\{ (x,m,n)\in \bR\times \bZ \times \bZ: |x|\le 2l, m\ne 0, n \ne 0, |m-k|\lea N, |n|\lea N,  \l| l x + mn  +C  \r|\lea 1         \r\}        . 
\end{align*} 
Then $|\mathcal{B}|\lea   \l( \frac{M}{N}\r)^{4\delta} N$, uniformly in $l\in\R$ with $0\leq l\lesssim M^{1-4\delta}N^{4\delta}$, $k\in\bZ$, $C\in\R$, and 
$1\le M\leq N$.
\end{lemma}
\begin{proof}

{\bf Case a)}: $l\lea 1$. Note that 
$$  \l| l x + mn+C  \r|\lea 1  \Longrightarrow \l|mn+C  \r|\lea 1.$$
Choose $\varepsilon\in (0, \frac12)$. Using Lemma \ref{lem:counting lem}, we have
$$ |\mathcal{B}|\lea l \cdot \l| \{  (m,n)\in\bZ^2: \l|mn+C  \r|\lea 1, m\ne 0, n \ne 0, |m-k|\lea N, |n|\lea N \} \r| \lea N^\varepsilon \lea \l( \frac{M}{N}\r)^{4\delta} N.$$

{\bf Case b)}: $1\ll l\lea N^{\frac12}$. Note that  
$$  \l| l x + mn+C  \r|\lea 1  \Longrightarrow \l|mn+C  \r|\lea l^2.$$
Also note that 
$$|\{x\in\R: \l| l x + mn+C  \r|\lea 1\}|\lesssim \frac1l.$$
Choose $\varepsilon \in (0, \frac12-4\delta)$. By Lemma \ref{lem:counting lem} again,
$$ 
\l| \{  (m,n)\in\bZ^2: \l|mn+C  \r|\lea l^2, m\ne 0, n \ne 0, |m-k|\lea N, |n|\lea N  \} \r| \lea  l^2 N^\varepsilon,
$$
thus
\begin{align*}
  |\mathcal{B}| &\lea \frac{1}{l} \cdot \l| \{  (m,n)\in\bZ^2: \l|mn+C  \r|\lea l^2, m\ne 0, n \ne 0, |m-k|\lea N, |n|\lea N  \} \r| \\
      &\lea l N^\varepsilon \lea N^{\frac12+\varepsilon} \lea \l( \frac{M}{N}\r)^{4\delta} N.
\end{align*}

{\bf Case c)}: $N^{\frac12}\ll l\lea M^{1-{4\delta}} N^{{4\delta}}$.
Then $l^2\gg N$.  Note that
$$  \l| l x + mn+C  \r|\lea 1  \Longrightarrow \l|mn+C  \r|\lea l^2.$$
We employ an alternative approach to estimating
$$  \l| \{  (m,n)\in\bZ^2: \l|mn+C  \r|\lea l^2, m\ne 0, n \ne 0, |m-k|\lea N, |n|\lea N  \} \r|.  $$
First, assume that $k\lea N$. Then 
\begin{align*}
  & \l| \{  (m,n)\in\bZ^2: \l|mn+C  \r|\lea l^2, m\ne 0, n \ne 0, |m-k|\lea N, |n|\lea N  \} \r| \\
\lea & \l| \{  (m,n)\in\bZ^2: \l|mn+C  \r|\lea l^2, m\ne 0, n \ne 0, |m|\lea N, |n|\lea N  \} \r|.
\end{align*}
By considering the cases $|m|\ge |n|$ and $|m| < |n|$ separately, we obtain
\begin{align*}
  & \l| \{  (m,n)\in\bZ^2: \l|mn+C  \r|\lea l^2, m\ne 0, n \ne 0, |m|\lea N, |n|\lea N  \} \r|\\
\lea & \sum_{1\le |m| \lea N } \min\l\{ \frac{l^2}{|m|}+1, |m|      \r\} + \sum_{1\le |n| \lea N } \min\l\{ \frac{l^2}{|n|}+1, |n|      \r\} \\
\lea & l^2\max\l\{1,\log\l(\frac{N}{l} \r)\r\}.
\end{align*}
The remaining case is when $k\gg N$. Then $|m-k|\lea N \ll k$ implies $m\ge |m-k|$. By considering the cases $|m-k|\ge |n|$ and $|m-k| < |n|$ separately, we obtain
\begin{align*}
 & \l| \{  (m,n)\in\bZ^2: \l|mn+C  \r|\lea l^2, m\ne 0, n \ne 0, |m-k|\lea N, |n|\lea N  \} \r|\\
\lea & \sum_{1\le |m-k| \lea N } \min\l\{ \frac{l^2}{m}+1, |m-k|      \r\} + \sum_{1\le |n| \lea N } \min\l\{ \frac{l^2}{|n|}+1, |n|      \r\} \\
\lea & \sum_{1\le |m-k| \lea N } \min\l\{ \frac{l^2}{|m-k|}+1, |m-k|      \r\} + \sum_{1\le |n| \lea N } \min\l\{ \frac{l^2}{|n|}+1, |n|      \r\} \\
\lea & l^2 \max\l\{1,\log\l(\frac{N}{l} \r)\r\}. 
\end{align*}

To conclude, we have
\begin{align*}
  |\mathcal{B}| &\lea \frac{1}{l} \cdot \l| \{  (m,n): \l|mn+C  \r|\lea l^2, m\ne 0, n \ne 0, |m-k|\lea N, |n|\lea N  \} \r| \\
      &\lea l \max\l\{1,\log\l(\frac{N}{l} \r)\r\}\\
      &\lea  \l( \frac{M}{N}\r)^{4\delta} N,
\end{align*}
where in the last inequality, we used ${4\delta} \in (0, \frac{1}{2}) $ and $l\lea M^{1-{4\delta}} N^{{4\delta}}$.   This finishes the proof.    
\end{proof}

\begin{remark}\label{rmk 5.1}
In the proof of Case 2 for Theorem \ref{mainthm: 2}, it is pivotal that we decompose the kernel function into $K_1$ and $K_2$ as in \eqref{eq: K1 + K2}, which allows us to leverage different arithmetic--geometric mean inequalities as in \eqref{eq: AM-GM 1} and \eqref{eq: AM-GM 2}. Previous approaches such as in \cite{TT01, DPST07, HTT14, DFYZZ24}, have been essentially using \eqref{eq: AM-GM 1} only. This technique is 
robust and likely useful for addressing other multilinear-type estimates. In particular, it can be applied to establish the \textit{sharp} $L^4$-Strichartz estimate for the hyperbolic Schr\"odinger equation on $\mathbb{R} \times \mathbb{T}$, 
\begin{equation}
\|e^{it(\partial_{xx}-\partial_{yy})}f\|_{L^4_{t,x,y}([0,1]\times \mathbb{R}\times\T)} \lesssim \|f\|_{L^2(\mathbb{R}\times\T)}\footnote{Motivated by \cite{Bar21,MR4219972}, a global-in-time version may also hold due to dispersion in the $\mathbb{R}$ direction.},
\end{equation}
which extends the classical work of Takaoka and Tzvetkov \cite{TT01} beyond the elliptic case.
This estimate yields well-posedness in the $L^2$-critical space for the cubic hyperbolic NLS on $\mathbb{R} \times \mathbb{T}$, and also has applications to related models such as the (hyperbolic--elliptic) Davey–Stewartson system, the KP-II equation, and the gravity
water waves. A detailed treatment will appear in a forthcoming work \cite{next}.

\end{remark}

\section{Refined bilinear Strichartz estimate and well-posedness of the energy-critical NLS 
}\label{bilinear and WP}
In this section, we derive the refined bilinear Strichartz estimate on $\mathbb{R} \times \mathbb{S}^3$ stated in Theorem \ref{mainthm: 3} from Theorems \ref{mainthm: 1} and \ref{mainthm: 2}, and then, as a standard corollary, we deduce the well-posedness theory in Theorem \ref{mainthm: 4}. In particular, to make Theorem \ref{mainthm: 2} applicable, we proceed in two steps: first, we employ the spatial and temporal almost orthogonality argument as in \cite{HTT11}; second, following \cite{Her13,HS15}, we apply the Plancherel identity on $\mathbb{R}_t \times \mathbb{R}_x$ before invoking the bilinear eigenfunction estimate on $\mathbb{S}^3$.

\subsection{Bilinear Strichartz estimate}

\begin{proof}[Proof of Theorem \ref{mainthm: 3}]

  It suffices to prove 
    \begin{align*}
        \|\varphi^2(t)e^{it\Delta}P_{N_1}f\cdot e^{it\Delta}P_{N_2}g\|_{L^2(\R\times\R\times\bS^3)}\lesssim N_2\left(\frac{N_2}{N_1}+\frac{1}{N_2}\right)^\delta\|f\|_{L^2(\R\times\bS^3)}\|g\|_{L^2(\R\times\bS^3)}.
    \end{align*}
    As mentioned in Section \ref{pre}, to ease notations, we add 1 to the spectra of $\Delta$, which amounts to redefining the Laplace--Beltrami operator on $\R\times \bS^3$ as  
    $$\Delta=\Delta_{\R}+\Delta_{\bS^3}-\mathrm{Id}.$$
    It suffices to prove the above estimate for this new $\Delta$. 
    Now we follow a strategy as in \cite{HTT11} and later followed by \cite{HTT14, Her13, HS15}, which explores almost orthogonality in both spatial and temporal directions. 
    Let us first assume that $N_2\ll N_1$. 
    We first perform a spectral localization which will pertain to the spatial almost orthogonality.  Partition $\R\times\bZ$ into a collection of disjoint cubes $C$ of side length $N_2$, so that 
    \begin{align}\label{eq: first localization}
        P_{N_1}f=\sum_C P_{C}P_{N_1}f.
    \end{align}
    It suffices to consider those $C$ such that 
    \begin{align}\label{eq: ring}
        C\cap \{(\omega,k+1)\in\R\times\bZ_{\geq 1}: \frac{N^4_1}{4}\leq (k+1)^2+\omega^2-1\leq 4N^2_1\}\neq \varnothing.
    \end{align}   
    
    By Lemma \ref{lem: fg spec supp}, $P_CP_{N_1}f\cdot P_{N_2}g$ is spectrally supported in $C+[-2N_2,2N_2]^2$. This implies that $P_CP_{N_1}f\cdot P_{N_2}g$ are an almost orthogonal family in $L^2(\R\times\bS^3)$ over the $C$'s. 
    Thus it suffices to prove 
    \begin{align}\label{eq: with varphi(t) and C}
        \|\varphi^2(t)e^{it\Delta}P_{C}P_{N_1}f\cdot e^{it\Delta}P_{N_2}g\|_{L^2(\R\times\R\times\bS^3)}\lesssim N_2\left(\frac{N_2}{N_1}+\frac{1}{N_2}\right)^\delta\|f\|_{L^2(\R\times\bS^3)}\|g\|_{L^2(\R\times\bS^3)},
    \end{align}
    uniformly over $C$. 
    
   We perform the second spectral localization which will pertain to the temporal almost orthogonality. Let 
   $$M=\max\left\{\frac{N_2^2}{N_1},1\right\}.$$
   Then $M\leq N_2$. 
   We partition $C$ into slabs. Let $\xi^0$ denote the center of $C$. Because of \eqref{eq: ring} and  $N_2\ll N_1$, we have 
   $$|\xi^0|\sim N_1.$$
   Let $a=\xi^0/|\xi^0|$. Write 
   \begin{align}\label{eq: second localization}
       P_C P_{N_1}f=\sum_{\mathcal{R}} P_\mathcal{R}f,
   \end{align}   
   where each $\mathcal{R}$ is of the form 
   \begin{align}\label{eq: slab}
          \mathcal{R}=\{\xi\in C: |a\cdot\xi-c|\leq M\},
   \end{align}  
   in which $c\in 2M\cdot\mathbb{Z}$. 
    Again, because of \eqref{eq: ring} and  $N_2\ll N_1$, 
   it follows that $|c|\sim N_1$. 
   The temporal frequency of $e^{it\Delta}P_\mathcal{R}f$ corresponding to the spectral parameter $\xi\in \mathcal{R}$, is 
    \begin{align*}
        -|\xi|^2&=-(\xi\cdot a)^2-|\xi-(\xi\cdot a)a|^2\\
        &=-c^2-(\xi\cdot a-c)^2-2c(\xi\cdot a-c)-|\xi-(\xi\cdot a)a|^2\\
        &=-c^2+  O(M^2+cM+N_2^2).
    \end{align*}
Since $|c|\sim N_1\gg N_2\geq M$, and $N_2^2\lesssim N_1M\sim |c|M$, 
   we have 
   $$-|\xi|^2=-c^2+  O(cM).$$
    Now that the temporal frequency of $e^{it\Delta}P_{N_2}g$ is supported in $[-4N_2^2-1,4N^2_2+1]$, and that the frequency of $\varphi^2(t)$ is supported in $[-2,2]$, we conclude that the temporal frequency of the product $\varphi^2(t)e^{it\Delta}P_{\mathcal{R}}f\cdot e^{it\Delta}P_{N_2}g$ is still
    $$-c^2+ O(cM).$$
    This implies that $\varphi^2(t)e^{it\Delta}P_{\mathcal{R}}f\cdot e^{it\Delta}P_{N_2}g$ are an almost orthogonal family in $L^2_t(\R)$ over the slabs $\mathcal{R}$, as $c$ ranges in $2M\cdot\bZ$ with $|c|\gg M$. This further reduces \eqref{eq: with varphi(t) and C} to 
     \begin{align}\label{eq: with R}
        \|\varphi^2(t)e^{it\Delta}P_{\mathcal{R}}f\cdot e^{it\Delta}P_{N_2}g\|_{L^2(\R\times\R\times\bS^3)}\lesssim N_2\left(\frac{N_2}{N_1}+\frac{1}{N_2}\right)^\delta\|f\|_{L^2(\R\times\bS^3)}\|g\|_{L^2(\R\times\bS^3)},
    \end{align}
    uniformly over $\mathcal{R}$. 

Let $\rho(t)=\varphi^2(t)$. Then $\widehat{\rho}=\widehat{\varphi}*\widehat{\varphi}\geq 0$. 
Now we may write for $(t,x,y)\in\R\times\R\times\bS^3$, that 
\begin{align*}
&\varphi^2(t)e^{it\Delta}P_{\mathcal{R}}f(x,y)\cdot e^{it\Delta}P_{N_2}g(x,y)\\
=&\int_\R \widehat{\rho}(\tau)e^{it\tau}\ \l( \int_{\R\times\bZ_{\geq 0}} f_{\omega_1,k_1}(y)e^{ix\cdot \omega_1-it(\omega_1^2+(k_1+1)^2)}\ {\rm d}\omega_1\ {\rm d}k_1\r) \\
& \cdot\l(\int_{\R\times\bZ_{\geq 0}} g_{\omega_2,k_2}(y)e^{ix\cdot \omega_2-it(\omega_2^2+(k_2+1)^2)}\ {\rm d}\omega_2\ {\rm d}k_2\r) d\tau,
\end{align*}
where $f_{\omega_1,k_1}=0$ if $(\omega_1,k_1+1)\notin \mathcal{R}$, and $g_{\omega_2,k_2}=0$ if $\omega^2_2+(k_2+1)^2>4N_2^2+1$. In particular, we may assume 
$|k_2|\lesssim N_2$. 
Continuing, we have 
$$  \varphi^2(t)e^{it\Delta}P_{\mathcal{R}}f\cdot e^{it\Delta}P_{N_2}g=   \int_\R\int_\R F(\tau',\omega,y) e^{it\tau'+ix\omega}\ d\tau'\ {\rm d}\omega,                    $$
where 
$$F(\tau',\omega,y)=\int_\R\sum_{k_1=0}^\infty\sum_{k_2=0}^\infty\widehat{\rho}(\tau'+\omega_1^2+(k_1+1)^2+|\omega-\omega_1|^2+(k_2+1)^2)f_{\omega_1,k_1}(y)g_{\omega-\omega_1,k_2}(y)\ {\rm d}\omega_1.$$
As $f_{\omega,k_1}$ and $g_{\omega-\omega_1,k_2}$ are eigenfunctions of the Laplace--Beltrami operator on $\bS^3$ with eigenvalues $-(k_1+1)^2+1$ and $-(k_2+1)^2+1$ respectively, we may apply Theorem \ref{mainthm: 1} to get 
\begin{align*}
    \|f_{\omega_1,k_1}g_{\omega-\omega_1,k_2}\|_{L^2(\bS^3)}
&\lesssim \min\{k_1,k_2\}^{\frac12}\|f_{\omega_1,k_1}\|_{L^2(\bS^3)}\|g_{\omega-\omega_1,k_2}\|_{L^2(\bS^3)}\\
&\leq k_2^{\frac12}\|f_{\omega_1,k_1}\|_{L^2(\bS^3)}\|g_{\omega-\omega_1,k_2}\|_{L^2(\bS^3)}\\
&\lesssim N_2^{\frac12}\|f_{\omega_1,k_1}\|_{L^2(\bS^3)}\|g_{\omega-\omega_1,k_2}\|_{L^2(\bS^3)}. 
\end{align*}
An application of Minkowski's inequality then yields
\begin{align*}
 &\|F(\tau',\omega,y)\|_{L^2_y(\bS^3)}\\
\lesssim &N_2^{\frac12}
\int_\R\sum_{k_1}\sum_{k_2}\widehat{\rho}(\tau'+\omega_1^2+(k_1+1)^2+|\omega-\omega_1|^2+(k_2+1)^2)\|f_{\omega_1,k_1}\|_{L^2(\bS^3)}\|g_{\omega-\omega_1,k_2}\|_{L^2(\bS^3)}\ {\rm d}\omega.   
\end{align*}
Then by the Plancherel identity for $\R^2$, we have 
\begin{align}
 &\|\varphi^2(t)e^{it\Delta}P_{\mathcal{R}}f\cdot e^{it\Delta}P_{N_2}g\|_{L^2(\R\times\R\times\bS^3)} \\
 =&\l\|\|F(\tau',\omega,y)\|_{L^2_y(\bS^3)}\r\|_{L^2_{\tau',\omega}}\nonumber\\
\lesssim& N_2^{\frac12}
\l\|\int_\R\sum_{k_1}\sum_{k_2}\widehat{\rho}(\tau'+\omega_1^2+(k_1+1)^2+|\omega-\omega_1|^2+(k_2+1)^2)\|f_{\omega_1,k_1}\|_{L^2(\bS^3)}\|g_{\omega-\omega_1,k_2}\|_{L^2(\bS^3)}\ {\rm d}\omega\r\|_{L^2_{\tau',\omega}}.\nonumber 
\end{align}
Using the Plancherel identity for $\R^2$ again, and applying H\"older's inequality, we further bound the above by 
\begin{align}
N_2^{\frac12}&\l\|\rho(t) \l(\int_\R\sum_{k_1} \|f_{\omega_1,k_1}\|_{L^2(\bS^3)}e^{ix\cdot \omega_1-it(\omega_1^2+(k_1+1)^2)}\ {\rm d}\omega_1 \r)\r. \\
&\ \ \cdot \l.\l(
\int_\R\sum_{k_2} \|g_{\omega_2,k_2}\|_{L^2(\bS^3)}e^{ix\cdot \omega_2-it(\omega_2^2+(k_2+1)^2)}\ {\rm d}\omega_2 \r)\r\|_{L^2_{t,x}}\nonumber \\
\lesssim N_2^{\frac12}&\l\|\varphi(t)\int_\R\sum_{k_1} \|f_{\omega_1,k_1}\|_{L^2(\bS^3)}e^{ix\cdot \omega_1-it(\omega_1^2+(k_1+1)^2)}\ {\rm d}\omega_1\r\|_{L^4_{t,x}} \label{eq: y removed}\\
\cdot &\l\|\varphi(t)
\int_\R\sum_{k_2} \|g_{\omega_2,k_2}\|_{L^2(\bS^3)}e^{ix\cdot \omega_2-it(\omega_2^2+(k_2+1)^2)}\ {\rm d}\omega_2\r\|_{L^4_{t,x}}. 
\end{align}
For the first $L^4_{t,x}$ norm above, recall that 
$f_{\omega_1,k_1}=0$ unless $(\omega_1,k_1+1)$ lies in the slab $\mathcal{R}$ defined in \eqref{eq: slab}.  We have 
$$\mathcal{R}\subset \{\xi=(\omega,k+1)\in\R\times\bZ: |\xi-\xi^0|\leq N_2, |a\cdot\xi-c|\leq M\}.$$
Apply Theorem \ref{mainthm: 2}, we have for any $\delta\in (0,\frac{1}{8})$, 
\begin{align}
\l\|\varphi(t)\int_\R\sum_{k_1} \|f_{\omega_1,k_1}\|_{L^2(\bS^3)}e^{ix \omega_1-it(\omega_1^2+(k_1+1)^2)}\ {\rm d}\omega_1\r\|_{L^4_{t,x}}\nonumber 
    &\lesssim \left(\frac{M}{N_2}\right)^\delta N_2^\frac{1}{4}\l\|\|f_{\omega_1,k_1}\|_{L^2(\bS^3)}\r\|_{L^2_{\omega_1,k_1}} \nonumber\\
    &\lesssim \left(\frac{N_2}{N_1}+\frac{1}{N_2}\right)^\delta N_2^\frac{1}{4}\|f\|_{L^2(\bR\times\mathbb{S}^3)}. \label{eq: res f}
\end{align}

For the other $L^4_{t,x}$ norm, recall that $g_{\omega_2,k_2}=0$ whenever $\omega^2_2+(k_2+1)^2>4N_2^2+1$. Then 
we may estimate it via the anisotropic Strichartz estimate \eqref{eq: res} on $\R\times \T$: 
\begin{align}
    \l\|\varphi(t)
\int_\R\sum_{k_2} \|g_{\omega_2,k_2}\|_{L^2(\bS^3)}e^{ix\cdot \omega_2-it(\omega_2^2+(k_2+1)^2)}\ {\rm d}\omega_2\r\|_{L^4_{t,x}}
\lesssim& N_2^\frac{1}{4}\l\|\|g_{\omega_2,k_2}\|_{L^2(\bS^3)}\r\|_{L^2_{\omega_2,k_2}}  \nonumber\\
=&N_2^\frac{1}{4}\|g\|_{L^2(\R\times\bS^3)}. \label{eq: res g}
\end{align}

Combine \eqref{eq: y removed}, \eqref{eq: res g} and \eqref{eq: res f}, we finish the proof, at least for the case $N_2\ll N_1$. 
To prove the case $N_2\sim N_1$, the two spectral localizations as in \eqref{eq: first localization} and \eqref{eq: second localization} are not needed. It suffices to follow the rest of the argument in the above proof, which eventually reduces to an application of \eqref{eq: res}, as in \eqref{eq: res g}. This finally finishes the proof. 

\end{proof}

\begin{remark}\label{rem: RmSn}
    The above derivation of Theorem \ref{mainthm: 3} from Theorem \ref{mainthm: 1} and \ref{mainthm: 2} is robust and can be easily adapted to obtain bilinear and multilinear Strichartz estimates on other product manifolds such as $\R^m\times\bS^n$.     
    For example, 
    one can show for all $m\geq 1$, $n\geq 3$, and $1\leq N_2\leq N_1$, there exists $\delta>0$ such that 
    $$\|e^{it\Delta}P_{N_1}f\cdot e^{it\Delta}P_{N_2}g\|_{L^2([0,1]\times\R^m\times\bS^n)}\lesssim N_2^{\frac{d-2}{2}}\left(\frac{N_2}{N_1}+\frac{1}{N_2}\right)^\delta\|f\|_{L^2(\R^m\times\bS^n)}\|g\|_{L^2(\R^m\times\bS^n)},$$
    where $d=m+n$ is the dimension of the product manifold. The case $(m,n)=(1,3)$ is of particular interest because of its energy-critical nature, and it also presents the greatest difficulty. Indeed, if $n\geq 4$, then the analogue of Theorem \ref{mainthm: 1} was already established in \cite{BGT052}; while if $m\geq 2$, the analogue of Theorem \ref{mainthm: 2} follows easily from the sharp Strichartz estimates on $\R^m\times\T$ obtained in \cite{Bar21}.
    For trilinear estimates, one can also show for all $m\geq 1$, $n\geq 2$, and $1\leq N_3\leq N_2\leq N_1$, that 
    \begin{align*}
        &\|e^{it\Delta}P_{N_1}f\cdot e^{it\Delta}P_{N_2}g\cdot e^{it\Delta}P_{N_3}h\|_{L^2([0,1]\times\R^m\times\bS^n)}\\
        \lesssim & (N_2N_3)^{\frac{d-1}{2}}\left(\frac{N_3}{N_1}+\frac{1}{N_2}\right)^\delta\|f\|_{L^2(\R^m\times\bS^n)}\|g\|_{L^2(\R^m\times\bS^n)}\|h\|_{L^2(\R^m\times\bS^n)}.
    \end{align*}  
    The above estimates then lead to the same local well-posedness as in Theorem \ref{mainthm: 4} for the corresponding cubic or quintic NLS. See also \cite{Zha21,Zha25} for related results on various compact product manifolds. 
\end{remark}

\subsection{Well-posedness: Proof of Theorem \ref{mainthm: 4}}

We briefly recall the function spaces $U^p$ and $V^p$ introduced in \cite{KT05} (see also \cite{MR1827277}),  which have been successfully employed in the context of nonlinear Schr\"odinger equations on manifolds as in \cite{HTT11, HTT14, Her13, HS15}. In the following, we use $\mathcal{M}$ to denote $\R\times\bS^3$. 

\begin{definition}[$U^p$ spaces]
Let $1\leq p < \infty$. A $U^p$-atom is a piecewise defined function $a:\mathbb{R} \rightarrow L^2(\mathcal{M})$ of the form
\begin{align*}
a=\sum_{k=1}^{K-1}\chi_{[t_{k-1},t_k)}\phi_{k-1}
\end{align*}
where  $-\infty<t_0<t_1<\ldots<t_K\leq \infty$, and $\{\phi_k\}_{k=0}^{K-1} \subset L^2(\mathcal{M})$ with $\sum_{k=0}^{K-1}\|\phi_k\|^p_{L^2(\mathcal{M})}=1$.

\noindent The atomic space $U^p(\mathbb{R};L^2(\mathcal{M}))$ consists of all functions $u:\mathbb{R}\rightarrow L^2(\mathcal{M})$ such that $u=\sum_{j=1}^{\infty}\lambda_j a_j$ for $U^p$-atoms $a_j$, $\{\lambda_j\} \in l^1$, with norm
\begin{align*}
\|u\|_{U^p(\mathbb{R},L^2(\mathcal{M}))}:=\inf\l\{\sum^{\infty}_{j=1}|\lambda_j|:u=\sum_{j=1}^{\infty}\lambda_j a_j,\lambda_j\in \mathbb{C}, \quad a_j \text{ are } U^p\text{-atoms}\r\}.
\end{align*}
\end{definition}

\begin{definition}[$V^p$ spaces]
Let $1\leq p < \infty$. We define $V^p(\mathbb{R},L^2(\mathcal{M}))$ as the space of all functions $v:\mathbb{R} \rightarrow L^2(\mathcal{M})$ such that
\begin{align*}
\|v\|_{V^p(\mathbb{R},L^2(\mathcal{M}))}:=\sup\limits_{-\infty<t_0<t_1<\ldots<t_K\leq \infty}\l(\sum_{k=1}^{K}\|v(t_k)-v(t_{k-1})\|^p_{L^2(\mathcal{M})}\r)^{\frac{1}{p}} < +\infty,
\end{align*}
where we use the convention $v(\infty)=0$. Also, we denote the closed subspace of all right-continuous functions $v:\mathbb{R}\rightarrow L^2(\mathcal{M})$ such that $\lim\limits_{t\rightarrow -\infty}v(t)=0$ by $V^p_{rc}(\mathbb{R},L^2(\mathcal{M}))$. 
\end{definition}

\begin{definition}[$X^s$ and $Y^s$ norms]\label{def: Xs Ys}
Let $s\in \mathbb{R}$. We define $X^s$ as the space of all functions $u:\R\to L^2(\mathcal{M})$, such that for all $N=2^m$, $m\geq 0$, the map $t\mapsto e^{-it\Delta}P_{N}u$ is in $U^2(\mathbb{R},L^2(\mathcal{M}))$, and for which the norm 
\begin{align*}
\|u\|_{X^s}^2=\sum_{N=2^m\geq 1} N^{2s} \|e^{-it\Delta}P_{N} u\|_{U^2(\R,L^2(\mathcal{M}))}^2
\end{align*}
is finite. We define $Y^s$ as the space of all functions $u:\R\to L^2(\mathcal{M})$, such that for all $N=2^m$, $m\geq 0$, the map $t\mapsto e^{-it\Delta}P_{N}u$ is in $V^2_{rc}(\mathbb{R},L^2(\mathcal{M}))$, and for which the norm 
\begin{align*}
\|u\|_{Y^s}^2=\sum_{N=2^m\geq 1} N^{2s} \|e^{-it\Delta}P_{N} u\|_{V^2(\mathbb{R},L^2(\mathcal{M}))}^2
\end{align*}
is finite. As usual, for a time interval $I\subset \R$, we also consider the restriction spaces $X^s(I)$ and $Y^s(I)$ defined in the standard way. 

\end{definition}

\begin{proposition}
    For $1\leq N_2\leq N_1$ and $0<\delta<\frac{1}{8}$, we have 
    $$\|P_{N_1}\widetilde{u_1}\cdot P_{N_2}\widetilde{u_2}\|_{L^2([0,1]\times \mathcal{M})}\lesssim N_2\left(\frac{N_2}{N_1}+\frac{1}{N_2}\right)^\delta\|P_{N_1}u_1\|_{Y^0}\|P_{N_2}u_2\|_{Y^0},$$
where $\widetilde{u_j}$ denotes either $u_j$ or $\overline{u_j}$.
\end{proposition}
\begin{proof}
The proof follows the same argument as in the derivation of Proposition 3.3 from Proposition 2.6 in \cite{HS15}, with only the trivial modification needed to pass from trilinear to bilinear estimates. We would like to only mention that Bernstein's inequalities were used, and we provided those in Lemma \ref{lem: Bernstein}. 
\end{proof}

For $f\in L^1_{\text{loc}}L^2([0,\infty)\times \mathcal{M})$, let 
$$\mathscr{I}(f)=\int_0^t e^{i(t-s)\Delta}f(s)\ {\rm d}s.$$ 
By arguments identical to the proof Proposition 2.12 in \cite{HTT14}, the above proposition yields the following nonlinear estimate of the Duhamel term.   

\begin{proposition}
    Let $s\geq 1$ be fixed. Then, for $u_1,u_2,u_3\in X^s([0,1))$, it holds 
    $$\left\|\mathscr{I}\left(\prod_{k=1}^3\widetilde{u_k}\right)\right\|_{X^s([0,1))}\lesssim \sum_{j=1}^3 \|u_j\|_{X^s([0,1))}\prod_{\substack{k=1\\k\neq j}}^3\|u_k\|_{X^1([0,1))}.$$
\end{proposition}

Theorem \ref{mainthm: 4} now follows from the above proposition in the usual way; see \cite{HTT11, HTT14, Her13, HS15, KV16}. More precisely, one can follow the derivation of Theorem 1.1 from Proposition 4.1 in \cite{HTT11} verbatim, with only the trivial modification required to pass from the energy-critical quintic NLS in three dimensions to the energy-critical cubic NLS in four dimensions. We would like to only mention that both the Bernstein and Sobolev inequalities were used, and we provided those in Lemma \ref{lem: Bernstein} and Lemma \ref{lem: Sob}.

\section{Open problems}\label{7}

We conclude by discussing several natural open problems that arise directly from our work.

\subsection {Anisotropic Strichartz estimate on $\R_{x_1}\times\bT_{x_2}$ of $L^\infty_{x_2}L^p_{t,x_1}$-type}
We make the following conjecture.
\begin{conj}\label{conj: Strichartz R times T}
Let $N\geq 1$. Then for all $p\geq 2$, it holds 
\begin{align}\label{eq: conj restricted Strichartz}
   &  \sup_{\xi^0\in\R\times\bZ}\l\|\varphi(t)\int_{
    \substack{\xi\in\bR \times \bZ \\
    |\xi-\xi^0|\leq N}} e^{ix_1\cdot \xi_1-it|\xi|^2} \widehat{\phi}(\xi)\ {\rm d}\xi\r\|_{L^p_{t,x_1}(\bR\times \bR)} \lea 
   (N^{1-\frac{3}{p}}+1)
   \|\phi\|_{L^2(\R\times\bT)}.
\end{align}
\end{conj}
The case $p=\infty$ follows from the Cauchy--Schwarz inequality. The $p=4$ case is also true, as mentioned in Remark \ref{rem: LinftyL4}. The $p=2$ case follows from a simple argument using the Plancherel identity for $\R\times\R$. By interpolation, we are missing the $2<p<4$ part of the above conjecture, which would follow from the critical case $p=3$. 
We mention that the above conjecture, if true, is sharp. The bound $N^{1-\frac3p}$ is seen to be saturated by testing against $\widehat{\phi}=\ca_{[-N,N]^{2}}$ and evaluating the $L^p_{t,x_1}([0,\frac{c}{N^2}]\times [0,\frac{c}{N}])$ norm, for some fixed small $c$. 

Interpolation between the obvious estimate
\begin{align*}
   &  \l\|\int_{\xi\in\mathcal{R}} e^{ix_1\cdot \xi_1-it|\xi|^2} \widehat{\phi}(\xi)\ {\rm d}\xi\r\|_{L^\infty_{t,x_1}(\bR\times \bR)} \lea \l(MN\r)^{\frac12}\|\phi\|_{L^2(\R\times\bT)},
\end{align*}
and the conjectured \eqref{eq: conj restricted Strichartz} for $p\in[3,4)$, would yield the refined anisotropic Strichartz estimate 
\begin{align*}
   &  \l\|\varphi(t)\int_{\xi\in\mathcal{R}} e^{ix_1\cdot \xi_1-it|\xi|^2} \widehat{\phi}(\xi)\ {\rm d}\xi\r\|_{ L^4_{t,x_1}(\bR\times \bR)} \lea
   \l(\frac{M}{N}\r)^{\frac12-\frac p8 } N^{\frac14}   \|\phi\|_{L^2}.
\end{align*}
Note that $p>3$ is equivalent to $\delta:=\frac12-\frac p8<\frac18$, which is exactly the range of $\delta$ covered by Theorem \ref{mainthm: 2}. Thus, Theorem \ref{mainthm: 2} may be viewed as positive evidence toward Conjecture  \ref{conj: Strichartz R times T}, except at the endpoint 
$p=3$. 

In the larger picture, Conjecture \ref{conj: Strichartz R times T} pertains to the pointwise behavior of the linear Schr\"odinger flow. The complication arises from the lack of dispersion in the compact $\bT$ factor. If we replace $\T$ with $\R$, and thus consider the analogous estimate on $\R^2$ corresponding to \eqref{eq: conj restricted Strichartz}, then it is not hard to show that this estimate holds for all $p\geq 2$---this follows by combining the dispersive estimates for the linear Schr\"odinger flow on both $\R_{x_1}\times\R_{x_2}$ and $\R_{x_2}$, the $TT^*$ argument, and the Hardy--Littlewood--Sobolev inequality. However, the problem becomes even more delicate on $\T^2$, where proving an analogue of Theorem \ref{mainthm: 2} would lead to the well-posedness result for the energy-critical NLS on $\T\times \bS^3$, as described in the Introduction.

\subsection{Strichartz estimate on $\R\times\bS^3$}
\label{subsec: Strichartz on RS3}
We make the following conjecture.
\begin{conj}\label{conj: Strichartz R S3}
 There holds the following Strichartz estimate on $\R\times\bS^3$
\begin{equation}
    \|e^{it\Delta}P_Nf\|_{L^p([0,1]\times \R\times\bS^3)}\lesssim N^{\sigma(p)}\|f\|_{L^2(\mathbb{R}\times \mathbb{S}^3)},
\end{equation}
for 
\[ 
\sigma(p) = 
\begin{cases}
2-\frac{6}{p}, & \text{if } p \geq \frac{10}3, \\
\frac12-\frac1{p},   & \text{if } 2 \leq p \leq \frac{10}3.
\end{cases}
\]
\end{conj}

The $p=\infty$ case as usual follows from Bernstein's inequality as in Lemma \ref{lem: Bernstein}. The $p=4$ case was provided in \eqref{eq: L4 Strichartz R times T}. 
This conjecture, if true, is also sharp.
The ``scale-invariant'' bound corresponding to $\sigma(p)=2-\frac{6}{p}$ is seen to be saturated by functions of the product form 
$$f(x,y)=g(x)\cdot h(y),\ x\in \R, \ y\in\bS^3, $$
for which we take $\widehat{g}(\omega)=\frac1{\sqrt{N}}\ca_{[-N,N]}(\omega)$, $\omega\in\R$, and take for a fixed $y_0\in\bS^3$,
$$h(y)=\sum_{j}N^{-\frac32}\beta\l(\frac{\lambda_j}{N}\r)e_j(y)\overline{e_j(y_0)},\ y\in\bS^3,$$
where $(\lambda_j)$ is the sequence of growing eigenvalues of $\sqrt{-\Delta_{\bS^3}}$ counted with multiplicities, $(e_j)$ is a corresponding orthonormal sequence of eigenfunctions, and $\beta$ is a bump function in $C^\infty_0((\frac12,2))$. We refer to the last section of \cite{HS25} for a detailed computation. The other bound corresponding to $\sigma(p)=\frac12-\frac1p$, coincides with the $2\leq p\leq 4$ piece of Sogge's $L^p$ bound for eigenfunctions of $\bS^3$, and to saturate the Strichartz bound it suffices to let $f$ be the highest weight spherical harmonics on $\bS^3$. 

Our formulation of Conjecture \ref{conj: Strichartz R S3} is primarily motivated and inspired by the recent work of Huang and Sogge in \cite{HS25}. They proved the Strichartz estimates on $\bS^2$, sharp up to $\varepsilon$-factors,
\begin{align*}
    \|e^{it\Delta_{\bS^2}}P_Nf\|_{L^p([0,1]\times \bS^2)}\lesssim_\varepsilon N^{\sigma(p)+\varepsilon}\|f\|_{L^2(\mathbb{S}^2)},
\end{align*}
for 
\[ 
\sigma(p)= 
\begin{cases}
1-\frac{4}{p}, & \text{ if } p \geq \frac{14}3, \\
\frac12\l(\frac12-\frac1p\r),   & \text{ if } 2 \leq p \leq \frac{14}3.
\end{cases}
\]
Analogous to Conjecture \ref{conj: Strichartz R S3}, the above Strichartz estimate comprises a scale-invariant regime alongside a regime saturated by the highest weight spherical harmonics.  This result was obtained using sophisticated tools, including microlocal analysis and bilinear oscillatory integral estimates. It remains an open question whether Conjecture \ref{conj: Strichartz R S3} can be established via similar techniques or by more elementary methods.
\begin{remark}\label{rem: R2S2}
The same Strichartz estimate as in Conjecture \ref{conj: Strichartz R S3} may be conjectured on $\T\times\bS^3$. 
Similarly, one may conjecture the following 
Strichartz estimate on 
$\R^m\times\T^{2-m}\times \bS^2$, $m=0,1,2$:
\begin{equation}
    \|e^{it\Delta}P_N f\|_{L^p([0,1]\times \R^m\times\T^{2-m}\times \bS^2)} 
    \;\lesssim\; N^{\sigma(p)} \|f\|_{L^2(\R^m\times\T^{2-m}\times \bS^2)},\footnote{Applying the Strichartz estimates on $\R^m\times \T^{2-m}$, together with Bernstein's inequality on $\bS^2$, we obtain that this Strichartz inequality holds on $\R^2\times\bS^2$, globally in time, for all $p\geq 4$; on $\R\times\T\times\bS^2$ for all $p\geq 4$; and on $\T^2\times\bS^2$ for all $p>4$.}
\end{equation}
with exponent
\[
\sigma(p) =
\begin{cases}
2-\tfrac{6}{p}, & p \geq \tfrac{22}{7}, \\[0.5em]
\tfrac{1}{4}-\tfrac{1}{2p}, & 2 \leq p \leq \tfrac{22}{7}.
\end{cases}
\]
An analogous estimate on the product space $\bS^2\times\bS^2$ would be
\begin{equation}
    \|e^{it\Delta}P_N f\|_{L^p([0,1]\times \bS^2\times\bS^2)} 
    \;\lesssim\; N^{\sigma(p)} \|f\|_{L^2(\bS^2\times\bS^2)},\footnote{Similarly, applying the Strichartz estimate together with Bernstein's inequality on $\bS^2$, we obtain that this Strichartz inequality on $\bS^2\times\bS^2$ holds for all $p\geq 6$.
}
\end{equation}
with
\[
\sigma(p) =
\begin{cases}
2-\tfrac{6}{p}, & p \geq \tfrac{10}{3}, \\[0.5em]
\tfrac{1}{2}-\tfrac{1}{p}, & 2 \leq p \leq \tfrac{10}{3}.
\end{cases}
\]
Note that the thresholds, $\tfrac{22}{7}$ and $\tfrac{10}{3}$, lie strictly below $4$. 
Although highly conjectural, such estimates---if valid---would lend support to critical 
well-posedness of the corresponding cubic NLS. See also \cite{Zha20,Zha21} for Strichartz estimates on general compact symmetric spaces. 
\end{remark}

\bibliographystyle{amsplain}
\bibliography{ref}
\end{document}